\documentclass{article}
\usepackage[T1]{fontenc}
\usepackage[english]{babel}
\usepackage{amsmath, amssymb, amsthm}
\usepackage{mathrsfs}
\usepackage{geometry}
\usepackage{graphicx}
\usepackage{fancyhdr} % intestazioni personalizzate
\usepackage{tikz}
\usetikzlibrary{decorations.markings,arrows.meta}
\usepackage{amscd}
\usepackage{epsf}
\usepackage{mathtools}
\usepackage{booktabs}
\usepackage{array}
\usepackage{multirow}
\usepackage{color}
\usepackage{tikz-cd}
\usepackage[symbol]{footmisc}
\usepackage{lipsum}
\usepackage{latexsym}
\usepackage{verbatim}
\usepackage{bbm}
\usepackage{setspace}
\usepackage[backref=page]{hyperref}
\renewcommand*{\backref}[1]{}
\renewcommand*{\backrefalt}[4]{%
	\ifcase #1 (Not cited)%
	\or        (Cited on page~#2)%
	\else      (Cited on pages~#2)%
	\fi}
\usepackage{etoolbox}
\usepackage[normalem]{ulem} % per \uline (sottolineatura)

\usepackage{enumitem}
\usepackage{anyfontsize}

\setlist[enumerate,1]{
  label=\textbf{(\arabic*)},
  ref=(\arabic*),
  leftmargin=*,
  itemsep=0.5ex,
  topsep=0.5ex
}

\let\OLDthebibliography\thebibliography
\renewcommand\thebibliography[1]{
  \OLDthebibliography{#1}
  \setlength{\parskip}{2pt}
  \setlength{\itemsep}{0pt plus 0.3ex}
}

\makeatletter
\renewenvironment{proof}[1][\proofname]{%
  \par\pushQED{\qed}%
  \normalfont\topsep6\p@\@plus6\p@\relax
  \trivlist
  \item[\hskip\labelsep
        \textbf{\textup{#1}}\@addpunct{.}]\ignorespaces
}{%
  \popQED\endtrivlist\@endpefalse
}
\makeatother
\newcommand{\mysubjclass}{\textup{2020} \textbf{Mathematics Subject Classification:}  53C55, 53D18, 53C25, 53D17.}
\newcommand{\mykeywords}{\textbf{Keywords: Generalized K\"ahler geometry;
Bismut--Hermitian--Einstein metrics;
pluriclosed solitons;
holomorphic Poisson structures;
K\"ahler orbifolds;
geometry with torsion.}
}

\fancypagestyle{firstpagefooter}{
 \fancyhf{} % pulisci header e footer
 \fancyfoot[L]{\scriptsize
 \mysubjclass \\
 \mykeywords \\
   }
}

\hypersetup{
    colorlinks=true,
    linkcolor=blue,    % Link per teoremi e equazioni in blu
    citecolor=red,     % Citazioni in rosso
    urlcolor=blue,     % Link URL in blu
    pdftitle={Generalized K\"ahler BHE threefolds},
    pdfauthor={Beatrice Brienza}
}

\theoremstyle{plain}
\newtheorem{thm}{Theorem}[section]
\newtheorem{introtheorem}{Theorem}[section]
\newtheorem{prop}[thm]{Proposition}

\newtheorem{defn}[thm]{Definition}
\newtheorem{cor}[thm]{Corollary}
\newtheorem{lem}[thm]{Lemma}
\newtheorem{claim}[thm]{Claim}

\theoremstyle{definition}
\newtheorem{rmk}[thm]{Remark}
\newtheorem{ex}[thm]{Example}
\numberwithin{equation}{section}

\newcommand{\Ric}{\mathrm{Ric}}

\newcommand{\R}{\mathbb{R}}

\newcommand{\SU}{\mathrm{SU}}

\renewcommand{\epsilon}{\varepsilon}
\renewcommand{\phi}{\varphi}

\renewcommand{\Im}{\mathrm{Im}}
\renewcommand{\div}{\mathrm{div}}
\DeclareMathOperator{\tr}{tr}
\DeclareMathOperator{\vol}{Vol}
\DeclareMathOperator{\grad}{grad}

\usetikzlibrary{shapes.misc}
\tikzset{cross/.style={cross out, draw, 
         minimum size=2*(#1-\pgflinewidth), 
         inner sep=0pt, outer sep=0pt},
         cross/.default={3.3pt}}

\def\eqref#1{(\ref{#1})}
\title{\textbf{Generalized K\"ahler BHE threefolds}}
\author{Beatrice Brienza}
\date{}

\begin{document}

\maketitle

\begin{abstract}
We study the notion of generalized K\"ahler
Bismut--Hermitian--Einstein (GKBHE) manifolds, namely
generalized K\"ahler manifolds for which both Bismut Ricci
forms vanish, and we develop a theory for  such structures on compact threefolds. We organize the analysis according to the horizontal rank, an invariant taking values in $\{0,1,2\}$ that measures
the relative position of the canonical foliations associated
with the two Hermitian structures. We characterize the case
of horizontal rank $0$ (which corresponds to K\"ahler Ricci flat geometry) and describe the case of horizontal
rank $1$, revealing its close interplay with the steady $\nu$-soliton equation in Kähler geometry. This description also provides a tool for constructing new compact GKBHE threefolds.
\end{abstract}

\thispagestyle{firstpagefooter}
\tableofcontents

\section{Introduction}

A Hermitian metric is \emph{Bismut--Hermitian--Einstein} (BHE)
if its Bismut connection has closed torsion and vanishing Ricci
form. BHE metrics are the natural non-K\"ahler analogue of
Calabi--Yau metrics
\cite{ABLS,ALL,ALLR,GFF,GFJS,GFStr,Rp2}: on compact manifolds
they are steady gradient solitons of the
pluriclosed flow \cite{GFJS,StreetsTian}, and they describe
supersymmetric string backgrounds \cite{FI,IP2,Str}. The theory
is rigid in low dimensions. Compact BHE surfaces are
classified \cite{GI}, % TODO: Gauduchon--Ivanov key
and on compact threefolds Apostolov, Barbaro, Lee and Streets
\cite{ABLS} showed that every non-K\"ahler BHE structure is
locally a principal $\mathbb R^2$-bundle over a K\"ahler surface
$(S,g_K)$ of positive scalar curvature $R_K$, whose Ricci form
$\rho_K$ satisfies the fourth-order equation
\begin{equation}\label{intro:box}
    \tfrac12\,dd^cR_K\wedge\omega_K
    =\rho_K\wedge\rho_K+4\,\alpha\wedge\alpha ,
\end{equation}
where $\alpha$ is the curvature of one of the two connection
forms. Equation \eqref{intro:box}, commonly called the \emph{Box equation} and already present in the
physics of $AdS_3\times Y_7$ backgrounds \cite{CAGMS,GAKI},
underlies all recent existence and topological results on BHE
threefolds \cite{ALL,ALLR,Steets}.
 
Generalized K\"ahler (GK) geometry offers a second, independent
extension of K\"ahler geometry. Introduced by Gates, Hull and
Ro\v cek \cite{GHR} as the target geometry of $(2,2)$
supersymmetric sigma models, and later recast by Gualtieri in
Hitchin's generalized geometry \cite{Cav,Gualtieri1,Gualtieri2,Hit},
a GK structure consists of two complex structures $I,J$,
orthogonal for the same metric $g$, whose Bismut connections
have closed torsions $H$ and $-H$. Its characteristic feature is
the real Poisson tensor $\sigma=\frac12[I,J]g^{-1}$, whose $(2,0)$ component with respect to either complex structure is holomorphic  \cite{AGG,Hit, Pont}.
 
This paper concerns the meeting point of the two theories.
 
\begin{defn}
A generalized K\"ahler manifold $(M,I,J,g)$ is
\emph{generalized K\"ahler BHE}, or \emph{GKBHE}, if both
$(I,g)$ and $(J,g)$ are BHE.
\end{defn}
 
\noindent
Equivalently, a GKBHE structure is a pair of BHE structures
    sharing the same metric and having opposite Bismut torsions. GKBHE structures have been considered previously, under the name of Bismut-Hermitian-Einstein pair, in ongoing work by Álvarez-Cónsul, De Arriba, Garcia-Fernandez and Van Ekeren \cite{ACDGFV} in relation to vertex algebras.
Besides Ricci-flat K\"ahler metrics, which are GKBHE for
$J=\pm I$, the basic example is the standard Hopf surface
$S^1\times S^3$, whose two Hermitian structures have flat Bismut connections.
Our goal is to understand to what extent the existence of the
second structure $J$ constrains the BHE geometry of $(I,g)$, and
conversely to construct compact GKBHE manifolds from data on the
quotient. We give a complete answer on compact threefolds in the
first two cases of a natural stratification.
\medskip

\noindent
\textbf{The horizontal rank.} Let $(M,I,J,g)$ be compact and GKBHE. Both $(I,g)$ and $(J,g)$
are gradient pluriclosed solitons, with the same normalized
potential $f$, and the vector fields
\[
    V_I=\tfrac12\bigl(\theta_I^\sharp-\grad f\bigr),
    \qquad
    V_J=\tfrac12\bigl(\theta_J^\sharp-\grad f\bigr),
\]
together with $IV_I$ and $JV_J$, are commuting Killing fields
encoding the a priori symmetries of the problem
(Lemmas~\ref{lem:a_priori_symmetries}--\ref{lem:norm} and
Proposition~\ref{prop:GKBHE-horizontal-symmetries}). If the
structure is not K\"ahler, $V_I$ and $IV_I$ span a
two-dimensional Riemannian foliation $\mathcal F_I$, which is the
vertical direction in the reduction of \cite{ABLS}. The
\emph{horizontal rank}
\[
    r=\max_{p\in M}\dim_{\mathbb R}
    \operatorname{span}\bigl\{
        \pi_{\mathcal H}V_J(p),\,
        \pi_{\mathcal H}JV_J(p)
    \bigr\}\in\{0,1,2\},
    \qquad \mathcal H=\mathcal F_I^\perp,
\]
measures how far the canonical foliation of $(J,g)$ is from
$\mathcal F_I$; we set $r=0$ in the K\"ahler case.
 
\begin{introtheorem}[cf.~Propositions~\ref{prop:rank0}
and~\ref{prop:linear_dependence}, Corollaries~\ref{cor:V_Iparallel}
and~\ref{cor:GKBHE-sign-involution}]
Let $(M,I,J,g)$ be a compact connected GKBHE manifold.
\begin{enumerate}
\item If $r=0$, then $(I,g)$ and $(J,g)$ are K\"ahler Ricci-flat.
\item If $r=1$, then $V_J=cV_I$ for some $c\in\{\pm1\}$, and
$V_I$ is parallel for the Levi-Civita connection.
\item If $r=1$, and $V_J=cV_I$ for some $c\in\{\pm1\}$, the bi-Hermitian structure $(I,-QJQ,g)$ is again GKBHE, where $Q$ is the reflection in
$V_I^\perp$. Furthermore
$V_{-QJQ}=-cV_I$.
\end{enumerate}
\end{introtheorem}

Thus rank-one structures come in pairs with the same metric and
the same $I$, exchanged by an involution reminiscent of
T-duality \cite{cavalcanti2010generalized,lindstrom2008t}. The
two members of the pair have different Poisson geometry. On
threefolds, normalizing $I$ and $J$ to induce the same
orientation, $\sigma$ vanishes identically for $c=1$, so that
$TM$ splits holomorphically
(Proposition~\ref{prop:vanishin_poisson}; cf.~\cite{AG}), whereas
for $c=-1$ it has rank four on a dense open set, with image
$\operatorname{span}\{V_I,IV_I,JV_I,IJV_I\}$. In both cases we
obtain a simultaneous normal form for $I$ and $J$
(Theorem~\ref{thm:normal_form_of_J}). The horizontal rank $2$ case will be explored in a future paper. 

\medskip
\noindent
\textbf{Rank one in dimension three: the transverse geometry.} When $r=1$, the curvature $\alpha$ in \eqref{intro:box}
vanishes, since $V_I^\flat$ is closed. The second BHE structure $(J,g)$,
however, imposes a much stronger, second-order constraint.
 
\begin{introtheorem}[cf.~Theorem~\ref{thm:transverse_geometry}]
Let $(M,I,J,g)$ be a compact connected GKBHE threefold with
$r=1$, $V_J=cV_I$ and $|V_I|=1$, and let $g_K$ be the transverse
K\"ahler metric. Then the basic function
\[
    F=e^{-f}\bigl(1+c\,g(IV_I,JV_J)\bigr)\ \ge 0
\]
is positive on a dense open set, and
\begin{equation}\label{intro:potential}
    I\grad_{g_K}F\ \text{is holomorphic},
    \qquad
    \rho_K=-dd^c\log F
    \quad\text{on }\{F>0\}.
\end{equation}
Moreover, $I\grad_{g_K}F$ is a constant multiple of the
horizontal part of $JV_J$.
\end{introtheorem}
 
In other words, $F$ is at the same time a Hamiltonian for a holomorphic
Killing field and, through its logarithm, a Ricci potential. 

We now turn to the converse, which is our main result. Its starting
point is a curvature identity valid on any K\"ahler surface: the
Ricci--Hamiltonian condition \eqref{intro:potential} already implies 
\eqref{intro:box} with $\alpha=0$
(Proposition~\ref{prop:kahler-identities}). Remarkably, the
fourth-order equation \eqref{intro:box} is a consequence of the
second-order system \eqref{intro:potential}. Positivity of $R_K$ is
not needed for this identity, but it is needed to go back up: when
$R_K>0$, the converse construction of \cite{ABLS}, with soliton
potential $f=\log(R_K/8)$, provides a BHE structure $(I,g)$ on any
principal $\mathbb T^2$-bundle with connection curvatures
$(0,-\tfrac12\rho_K)$. The central new step is to recover from the
same transverse data a second complex structure $J$ such that
$(I,J,g)$ is GKBHE. No further condition is
required: the Ricci--Hamiltonian equation \eqref{intro:potential} and $R_K>0$ suffice.
 
\begin{introtheorem}[cf.~Theorem~\ref{thm:global_converse}]\label{thm:introthm}
Let $(S,g_K)$ be a compact connected K\"ahler orbifold surface
with $R_K>0$, carrying a smooth function $F\ge0$, $F\not\equiv0$,
which solves \eqref{intro:potential}. Let $M\to S$ be a principal
$\mathbb T^2$-orbibundle with smooth total space and connection
forms $\eta,\eta_I$ such that $d\eta=0$ and
$d\eta_I=-\frac12\rho_K$. Then, for each $c\in\{\pm1\}$, $M$
carries a GKBHE structure $(I,J_c,g)$ of horizontal rank one with
$V_{J_c}=cV_I$, where $(I,g)$ is the BHE structure of
\cite{ABLS}.
\end{introtheorem}
 
The proof combines a local construction with an extension
argument across the zero set of $F$, which is necessarily
nonempty when the base is compact and has positive scalar
curvature.

We first work on $M_+:=p^{-1}(S_+)$, where
$S_+:=\{F>0\}$.
Starting from the BHE structure $(I,g)$, we use the normal
form established earlier to define two $g$-compatible
almost complex structures $J_c$, corresponding to
$c=\pm1$, whose expressions have apparent singularities
along the $M \setminus M_+$.

For $c=1$, the key step is an orthogonal splitting
$TM_+=L_+\oplus L_-$ into involutive subbundles with
K\"ahler leaves, satisfying $I=J_1$ on $L_+$ and
$I=-J_1$ on $L_-$.
The holomorphic properties of this splitting yield the
integrability of $J_1$, while the K\"ahler property of
the leaves and the sign relations between $I$ and $J_1$
give opposite Bismut torsions.
The remaining BHE identity $\rho^{B,J_1}=0$ follows from
the gradient pluriclosed soliton structure of $(I,g)$.
The involution $J_{-1}=-QJ_1Q$ produces the case $c=-1$,
completing the local converse
(Theorem~\ref{thm:local_converse}).

Finally, we prove that $M_+$ is dense and
use the Ricci-potential identity to remove the apparent
singularities in the formulas for $J_c$.
Hence $J_c$ extends smoothly across
$M \setminus M_+$, and the defining GKBHE identities
hold on the whole total space by continuity.

\medskip
\noindent
\textbf{Further remarks.} 
Equation \eqref{intro:potential} is closely related to weighted solitons. On $S_+=\{F>0\}$,
the identity $\rho_K=-dd^c\log F$, with $F$ a Hamiltonian
potential for a holomorphic Killing vector field, is precisely
the steady $\nu$-soliton equation with weight $\nu(x)=x^{-2}$.
For the compact shrinking (Fano) setting, see
\cite{Mabuchi2003,BermanWittNystrom2014,HanLi2023,ApostolovLahdiliNitta2025}. This is explained in detail in Remark \ref{rmk:nu_soliton}.

Furthermore, the geometry described by Equation \eqref{intro:potential} is reminiscent of the special
K\"ahler--Ricci potentials of Derdzi\'nski and Maschler
\cite{DM}, which however impose additional conditions on the
Hessian and the Ricci tensor. Where $F>0$ and $dF\ne0$, the
metric fits LeBrun's ansatz for K\"ahler surfaces with a
Hamiltonian Killing field \cite{LeBrun1991}, and
\eqref{intro:potential} yields the first integral
\[
    \tfrac12 R_KF+F^{-1}|\grad_{g_K}F|_{g_K}^{2}=\kappa>0 ,
\]
suggesting a relation with Toda-type equations. 
\medskip

In forthcoming
work \cite{BS}, Theorem \ref{thm:introthm} is used to produce compact GKBHE
threefolds diffeomorphic to $S^1\times\sharp_5(S^2\times S^3)$,
by resolving four out of eight orbifold points of
$(\mathbb{CP}^1\times E)/\mathbb Z_2$ with Eguchi--Hanson metrics
and solving an invariant Monge--Amp\`ere equation
(cf.~Proposition~\ref{prop:MA}), following a strategy proposed in
\cite{ALLR} to construct new examples of compact BHE threefolds.

\smallskip

{\bf Acknowledgments}: 
This work grew out of a conversation with Jeff Streets.
The author is deeply grateful to him for his continued
encouragement and many stimulating
discussions throughout the development of this project.
The author is also deeply grateful to Vestislav Apostolov for pointing out
the connection with steady $\nu$-solitons and for his helpful
comments regarding the manuscript. The author is also grateful to Anna Fino for kindly reading the first draft of this paper.
Editing and proof checking were carried out with the assistance of ChatGPT (OpenAI), access to which was generously funded by IMPA.

\section{Conventions, notation and standard definitions}
Let $(M^{2n},I,g)$ be a Hermitian manifold.
We use the following conventions. Throughout, $\{e_i\}_{i=1}^{2n}$
denotes a local $g$-orthonormal frame. 

\begin{itemize}
    \item
    We denote the Levi-Civita connection by $\nabla$, or by $\nabla^g$
    when the metric needs to be specified, and its curvature tensor
    by $\mathrm{R}^{\nabla}$.

    \item
    The Laplacian on functions is the negative of the Hodge Laplacian:
    \begin{equation} \label{eqn:laplacian}
       \Delta f=-d^*df
        =\div(\grad f)
        =\tr_g(\nabla df),   
    \end{equation}
    where $d^*$ denotes the Hodge codifferential.

    \item
    The complex structure $I$ acts on a $k$-form $\alpha$ by $(I\alpha)(X_1,\dots,X_k)=\alpha(IX_1,\dots,IX_k)$.

    \item
    We define $d^c=I^{-1}\circ d\circ I
        =i(\bar\partial-\partial)$.

    \item
    The fundamental form is defined by $\omega(X,Y)=g(IX,Y)$.

    \item For a $2$-form $\alpha$, we use the convention
    \[
    \operatorname{tr}_{\omega}\alpha
    :=
    \sum_{i=1}^{2n} \alpha(e_i,Ie_i).
    \]

    \item
    The \emph{Bismut connection}, denoted by $\nabla^{B,I}$, is the unique
    Hermitian connection with totally skew-symmetric torsion
    \cite{Bis,FI,Str}. It is given by
    \begin{equation}\label{eqn:bismut}
        g(\nabla^{B,I}_XY,Z)
        =g(\nabla_XY,Z)+\tfrac12 H(X,Y,Z),
    \end{equation}
    where $H(X,Y,Z)=-d^c\omega(X,Y,Z)=d\omega(IX,IY,IZ)$. $H$ will be called \emph{Bismut torsion}.

    \item
    The Hermitian structure $(I,g)$ is called \emph{pluriclosed}
    (or SKT) if $dH=0$, equivalently if
    $\partial\bar\partial\omega=0$.

    \item \label{item:Lee_form}
    The Lee form is the unique $1$-form $\theta$ satisfying $d\omega^{n-1}=\theta\wedge\omega^{n-1}$.
    In terms of the Bismut torsion, it is given by
    \begin{equation}\label{eqn:Lee_form}
        \theta(X)
        =-\tfrac12\sum_{i=1}^{2n}H(IX,e_i,Ie_i),
    \end{equation}
    see, for instance, \cite[Equation (2.9)]{IP2}.

    \item
    The following identity holds \cite{IP2}:
    \begin{equation}\label{eqn:torsion_trace}
        \sum_{i,j=1}^{2n}dH(e_i,Ie_i,e_j,Ie_j)
        =8|\theta|^2+8d^*\theta-\tfrac43|H|^2,
    \end{equation}
    where we use the tensor norm
    \[
        |H|^2
        =\sum_{i,j,k=1}^{2n}H(e_i,e_j,e_k)^2.
    \]

    \item
    The Hermitian structure $(I,g)$ is called \emph{balanced}
    if $\theta=0$, and \emph{conformally balanced} if
    $\theta=df$ for a globally defined smooth function $f$.
    In the latter case, the metric $e^{-f/(n-1)}g$ is balanced.
    When $n=2$, the balanced condition is equivalent to the
    K\"ahler condition, and hence conformally balanced metrics
    are precisely conformally K\"ahler metrics.

    \item
    Our curvature convention is $\mathrm{R}^{{B,I}}(X,Y)Z =\nabla^{B,I}_X\nabla^{B,I}_YZ-\nabla^{B,I}_Y\nabla^{B,I}_XZ-\nabla^{B,I}_{[X,Y]}Z.$
    The Ricci tensor is defined by  \begin{equation}\label{eqn:Ricci_tensor}
        \Ric^{{B,I}}(X,Y)
        =\sum_{i=1}^{2n}
        \mathrm{R}^{B,I}(e_i,X,Y,e_i)=\Ric^{\nabla}-\tfrac14 H^2-\tfrac12 d^*H,
    \end{equation}
    where
    \[
        H^2(X,Y)
        =\sum_{i,j=1}^{2n}
        H(X,e_i,e_j)H(Y,e_i,e_j).
    \]
    A proof of \eqref{eqn:Ricci_tensor} can be found for instance in \cite[Proposition 3.18]{GFStr}.
    \item
    The Bismut Ricci form is defined by
    \begin{equation}\label{eqn:rho_B}
        \rho^{{B,I}}(X,Y)
        =\tfrac12\sum_{i=1}^{2n}
        \mathrm{R}^{B,I}(X,Y,Ie_i,e_i).
    \end{equation}
    When $\rho^{{B,I}}$ vanishes, the Hermitian structure $(I,g)$ is usually called \emph{Calabi-Yau with torsion} \cite{GGP}. In this case, $\mathrm{Hol}^0(\nabla^{B,I}) \subseteq \SU(n)$.
\end{itemize}
\section{Preliminaries and general setting}\label{Sec:preliminaries}

We recall the basic notions of generalized K\"ahler and
Bismut--Hermitian--Einstein (BHE) geometry and fix our notation.
Throughout, $M$ denotes a manifold of real dimension $2n$.
For background on generalized K\"ahler geometry, we refer to
\cite{Cav,Gualtieri1,Gualtieri2,Hit}; for BHE geometry, see
\cite{ABLS,ALL,ALLR,GFJS,GFStr}.

\subsection{Generalized K\"ahler structures} \label{subs:GK}

Generalized K\"ahler geometry first arose
in the study of $N=(2,2)$ supersymmetric sigma models
\cite{GHR,HP}. Gualtieri \cite{Gualtieri1,Gualtieri2}
subsequently recast this geometry within the framework introduced
by Hitchin \cite{Hit}. In this formulation, a generalized K\"ahler
structure consists of two commuting generalized complex structures
$\mathcal{I}$ and $\mathcal{J}$ on an exact Courant algebroid,
with $-\mathcal{I}\mathcal{J}$ a generalized metric.
We adopt the equivalent bi-Hermitian description established
in \cite[Theorem~2.18]{Gualtieri2}.

\begin{defn}
A \emph{generalized K\"ahler structure} on $M$ is a bi-Hermitian structure $(I,J,g)$
such that $(I,g)$ and $(J,g)$ are pluriclosed structures
whose fundamental forms satisfy
\[
    d^c_I\omega_I=-d^c_J\omega_J.
\]
\end{defn}

Throughout, we denote the fundamental forms by $\omega_I$
and $\omega_J$, and set $ H:=-d^c_I\omega_I=d^c_J\omega_J$.
Thus the Bismut connections $\nabla^{B,I}$ and $\nabla^{B,J}$ have closed
torsion forms $H$ and $-H$, respectively. Note that the generalized K\"ahler
condition is invariant under $J\mapsto -J$, since
$\omega_{-J}=-\omega_J$ and $d^c_{-J}=-d^c_J$: when $n$ is odd, exchanging $J$ with $-J$ swaps the orientations. \par

The relation between generalized Kähler geometry and Poisson geometry is encoded by three natural
bivectors. The first is determined by the commutator of $I$ and $J$. 

\begin{prop}[\cite{AGG,Hit, Pont}]
\label{prop:commutator-poisson}
Let $(I,J,g)$ be a generalized K\"ahler structure. The bivector $\sigma:=\tfrac12[I,J]g^{-1}$
is real Poisson and has type $(2,0)+(0,2)$ with respect to both $I$ and $J$. Its $(2,0)$ component with respect to either complex structure is a holomorphic Poisson tensor.
\end{prop}
For any $p \in M$, we write
\[
    \operatorname{Im}\sigma_p
    :=\{\sigma_p(\alpha):\alpha\in T_p^*M\}=\operatorname{Im} \, [I,J]_p.
\]
The union of these spaces defines a singular distribution invariant under both $I$ and $J$. Let $\sigma^{2,0}$ be the $I$-holomorphic Poisson tensor induced by $\sigma$. At each point, 
\[
    \dim_{\mathbb R}(\operatorname{Im}\sigma_p)
    =2\dim_{\mathbb C}(\operatorname{Im}\sigma_{p}^{2,0}).
\]
Since the complex rank of a skew-symmetric tensor is even, the dimension of $\operatorname{Im}\sigma_p$ is divisible by four. In particular, in real dimension six, $\operatorname{Im}\sigma$ is a smooth distribution of rank four on $\{\sigma\neq 0\}$.

The other two Poisson tensors, firstly defined in \cite{LZ}, are
\begin{equation}\label{eq:poisson-plus-minus}
    \pi_\pm:=\tfrac12(I\pm J)g^{-1}.
\end{equation}
On the open sets $U_\pm:=\{p\in M:(I\pm J)_p\text{ is invertible}\}$, define $F_\pm:= - \pi_\pm^{-1}$. More precisely, $F_\pm (X,Y):=-2g((I\pm J)^{-1}X,Y)$ are symplectic forms tamed by $I$ \cite{ASU2}. 
We denote the corresponding symplectic volume forms by
\begin{equation} \label{eqn:symplectic_volume}
  \nu_\pm:=\frac{F_\pm^n}{n!}.   
\end{equation}
\begin{rmk}\label{rmk:symplectic-volume}
On the open set $U_+$, let $F_+=-\pi_+^{-1}$ and $dV_g=\frac{\omega_I^n}{n!}$. Then $F_+$ induces the complex orientation of $I$: this is a consequence of the fact that $F_+$ is tamed by $I$. Therefore, its symplectic volume form satisfies
\begin{equation}\label{eq:symplectic-volume}
    \nu_+:=\frac{F_+^n}{n!}
    =\frac{2^n}{\sqrt{\det(I+J)}}\,dV_g=2^n \, e^{-\psi_+} dV_g,
\end{equation}
where $\psi_+=\tfrac 12 \log (\det(I+J))$.
\end{rmk}
\begin{defn}\label{def:poisson-divergence}
Let $\mu$ be a volume form on $M$ and $\pi$ a bivector.
The \emph{divergence of $\pi$ with respect to $\mu$} is the
unique vector field $X=\operatorname{div}_\mu\pi$ satisfying
\begin{equation}\label{eq:poisson-divergence-definition}
    \iota_{X}\mu
    =d(\iota_\pi\mu).
\end{equation}
\end{defn}
\begin{claim}\label{claim:wedge_contraction}
Assume that $M$ is an oriented manifold with volume form $\mathrm{Vol}$.
For any bivector $\pi$ and any $1$-form $\alpha$: 
\begin{equation*}
\alpha \wedge \iota_\pi \vol= -\iota_{\pi(\alpha)} \vol.
\end{equation*}
\end{claim}
\begin{proof}
By bilinearity, we may assume that $\pi=X \wedge Y$. Then 
\[
\pi(\alpha)= X \wedge Y (\alpha)= \alpha(X)\, Y - \alpha(Y) \, X.
\]
Therefore 
\[
\iota_{\pi(\alpha)} \vol= \alpha(X)\, \iota_Y \vol  - \alpha(Y) \, \iota_X \vol.
\]
On the other hand,
\begin{align} 
0= \iota_\pi (\alpha \wedge \vol)= \iota_{Y} \iota_X(\alpha \wedge \vol)=\alpha(X) \iota_Y\vol - \alpha(Y) \iota_X\vol + \alpha \wedge \iota_\pi \vol.
\end{align}
The claim follows by these identities.
\end{proof}
The first identity in the following lemma is the computation underlying
\cite[Lemma~3.6]{ASU}. We include a proof with our sign conventions.
\begin{lem}\label{lem:divergence}
Let $(M^{2n},I,J,g)$ be a generalized K\"ahler manifold. Assume that $I$ and $J$ define the same orientation and let $dV_g=\frac{\omega_I^n}{n!}= \frac{\omega_J^n}{n!}$. For any $f\in C^\infty(M,\mathbb R)$, set
$\mu_f:=e^{-f}dV_g$. Then
\begin{equation}\label{eq:divergence}
    \operatorname{div}_{\mu_f}\pi_+
    =-\tfrac12\left(
        I\bigl((\theta_I-df)^\sharp\bigr)
        + J\bigl((\theta_J-df)^\sharp\bigr)
    \right).
\end{equation}
Furthermore, on $U_+$
\begin{equation}\label{eq:divergence_mu_nu}
    \operatorname{div}_{\mu_f}\pi_+
    = \pi_+ (df -d \psi_+),
\end{equation}
where $\psi_+=\tfrac12 \log(\det(I+J))$. 
\end{lem}

\begin{proof}
Fix $K\in\{I,J\}$ and write
$\omega_K^{[r]}:=\tfrac{\omega_K^r}{r!}$. For every $1$-form $\alpha$, since 
$\iota_{K\alpha^\sharp}\omega_K=-\alpha$, we have
\begin{equation}\label{eq:one-form-hermitian-volume}
    \alpha\wedge\omega_K^{[n-1]}
    =-\iota_{K\alpha^\sharp}dV_g.
\end{equation}
Equation \eqref{eq:one-form-hermitian-volume} and Claim \ref{claim:wedge_contraction} imply that for each $1$-form $\alpha$, 
\begin{equation}\label{eq:contraction-inverse-hermitian-form}
\alpha\wedge\omega_K^{[n-1]}=-\iota_{K\alpha^\sharp}dV_g=\alpha \wedge \iota_{Kg^{-1}} dV_g,
\end{equation}
and therefore $\iota_{Kg^{-1}} dV_g=\omega_K^{[n-1]}$.   
 
Using \eqref{eq:contraction-inverse-hermitian-form}, the definition of the Lee form and \eqref{eq:one-form-hermitian-volume}, we obtain
\[
\begin{aligned}
    d\bigl(\iota_{Kg^{-1}}\mu_{f}\bigr)
    &=d\bigl(e^{-f}\omega_K^{[n-1]}\bigr)\\
    &=e^{-f}(\theta_K-df)\wedge\omega_K^{[n-1]}\\
    &=-\iota_{K((\theta_K-df)^\sharp)}\mu_{f}.
\end{aligned}
\]
Thus
$\operatorname{div}_{\mu_f}(Kg^{-1})
=-K((\theta_K-df)^\sharp)$.
The first identity follows by linearity and
\eqref{eq:poisson-plus-minus}.

We now establish the second equality. Set $\rho:= \frac{\mu_f}{\nu_+} \in C^\infty(U_+)$, where $\nu_+$ is defined in Equation \eqref{eq:symplectic-volume}. Then $\log\rho=\log \tfrac{1}{2^{n}}+ \psi_+-f$, and hence $d\log \rho=d\psi_+-df$. We compute $d \iota_{\pi_+}\mu_f$. 
\begin{align*}
d \iota_{\pi_+}\mu_f &= d \iota_{\pi_+} (\rho\, \nu_+)\\
                   &= d\rho \wedge \iota_{\pi_+} \nu_+ + \rho \, d      \iota_{\pi_+} \nu_+ \\
                   &= d\rho \wedge \iota_{\pi_+} \nu_+\\
                   &=- \iota_{\pi_{+} (d\rho)} \nu_+\\
                   &=-\tfrac{1}{\rho} \iota_{\pi_{+} (d\rho)} \mu_f \\
                   &= \iota_{-\pi_{+} (d\log\rho)} \mu_f, 
\end{align*}
where in the third equality we used that $d\iota_{\pi_+} \nu_+=0$ since $F_+$ is symplectic, and the fourth equality follows by Claim \ref{claim:wedge_contraction}. By the definition of divergence, $\operatorname{div}_{\mu_f}\pi_+=-\pi_{+} (d\log\rho)=\pi_+(df-d\psi_+)$.
\end{proof}
\subsection{Bismut--Hermitian--Einstein structures}
\label{subsec:BHE-structures}
Bismut--Hermitian--Einstein structures, introduced in
\cite{GFJS,GFStr}, provide a natural non-K\"ahler counterpart of
Ricci-flat K\"ahler geometry. On compact manifolds, the theory is
particularly rigid in low dimensions: the surface case was
classified in \cite{GI}, the structure theory in real dimension six
was developed in \cite{ABLS}, and further topological restrictions
were established in \cite{Steets}.

\begin{defn}
Let $(M^{2n},I,g)$ be a pluriclosed Hermitian manifold.
\begin{enumerate}
    \item
    The Hermitian structure $(I,g)$ is called
    \emph{Bismut--Hermitian--Einstein}, or \emph{BHE}, if
    $\rho^{B,I}=0.$

    \item
    The Hermitian structure $(I,g)$ is called a \emph{ gradient
    pluriclosed soliton} if there exists a smooth function $f$ such that
    \begin{equation}\label{eq:-pluriclosed-soliton}
        \Ric^{\nabla}-\tfrac14 H^2+\nabla^2f=0,
        \qquad
        d^*H+\iota_{\grad f}H=0.
    \end{equation}
    When \(M\) is compact connected, the soliton potential \(f\) is unique up to an additive constant. We henceforth fix this ambiguity by requiring $\int_M e^{-f}\,dV_g=1,$ and call the resulting function the \emph{normalized soliton potential}.

    \item 
    A generalized K\"ahler structure $(g, I, J)$ is a \emph{generalized K\"ahler-Ricci soliton} if $(I,g)$ is a gradient
    pluriclosed soliton. In particular, also $(J,g)$ is a gradient
    pluriclosed soliton with the same function $f$. 
\end{enumerate}
\end{defn}
We prove the following Claim, which will be useful in Section \ref{section:converse}. 
\begin{claim} \label{claim:GRS}
Let $(M,I,J,g)$ be a generalized K\"ahler manifold, and suppose that
$(I,g,f)$ is a gradient pluriclosed soliton. Assume that
\[
\theta_I=df+2\eta,\qquad \theta_J=df+2c\eta,
\]
where $c\in\R$ and $\nabla \eta=\nabla^{B,I}\eta=0$. Then
$\rho^{B,J}=0$.
\end{claim}

\begin{proof}
The generalized K\"ahler condition implies that $(g,J,f)$ is also a gradient
 pluriclosed soliton. Since $\nabla \eta=\nabla^{B,I}\eta=0$ and the two Bismut torsions satisfy $H_J=-H$, 
$\nabla^{B,J}\eta=0$. Indeed, $\nabla=\nabla^{B,I}+\nabla^{B,J}$. Thus, by \cite[Proposition 2.17]{Lee1},
\[
\rho^{B,J}\circ J
=\nabla^{B,J}(\theta_J-df)
=2c\,\nabla^{B,J}\eta
=0.
\]
As $J$ is invertible, $\rho^{B,J}=0$.
\end{proof}
Note that the claim can be easily stated for generalized K\"ahler orbifold, by working in a uniformizing chart.

\medskip
Let $(M^{2n},I,g)$ be a compact BHE manifold. By
\cite[Proposition~2.6]{GFJS}, it is a  gradient pluriclosed
soliton. Denote by $f$ its normalized soliton potential and set
\begin{equation}\label{eq:canonical-BHE-vector-field}
    V_I:=\tfrac12\bigl(\theta_I^\sharp-\grad f\bigr).
\end{equation}

The structure results of \cite{ABLS,GFJS} give
\begin{equation}\label{eq:BHE-parallel-fields}
    \nabla^{B,I} V_I=0,
    \qquad
    \mathcal{L}_{V_I} I=0.
\end{equation}
In particular, $V_I$ and $I V_I$ are commuting real holomorphic
Killing vector fields of constant norm. Moreover, the soliton
potential is invariant under both vector fields, that is,
$V_I(f)=I V_I(f)=0$ \cite[Lemma~2.9]{ABLS}. Finally, $V_I$ vanishes
if and only if $(I,g)$ is K\"ahler. This was first proved in
\cite[Proposition~2.6]{GFJS}; a shorter proof was subsequently given
in \cite[Proposition~2.6]{ABLS}. In this case, the BHE condition
reduces to the Ricci-flat K\"ahler condition.

Suppose henceforth that $(I,g)$ is non-K\"ahler. We use the
normalization
\begin{equation}\label{eq:BHE-unit-normalization}
    |V_I|=|I V_I|=1.
\end{equation}
Then $\mathcal F_I:=\operatorname{span}\{V_I,I V_I\}$ defines a
rank-two holomorphic Riemannian foliation. Unless otherwise stated,
\emph{basic} will mean basic with respect to the foliation
$\mathcal{F}_I$. Introduce $\eta:=V_I^\flat,$ and
$\eta_I:=(I V_I)^\flat=-I\eta$. The pair $(\eta,\eta_I)$ defines an invariant connection for the foliation. Its
curvature forms $d\eta$ and $d\eta_I$ are basic and $I$-invariant
\cite[Lemma~2.7]{ABLS}. Thus every compact non-K\"ahler BHE manifold
is locally a $\R^2$-bundle over a Hermitian manifold of real
dimension $2n-2$.

Let $\mathcal H_I:=\ker\eta\cap\ker\eta_I=\mathcal F_I^\perp$ be the
horizontal distribution. The metric and the fundamental form decompose
as
\begin{equation}\label{eq:BHE-transverse-decomposition}
    g
    =\eta\otimes\eta
     +\eta_I\otimes\eta_I
     +g^T,
    \qquad
    \omega
    =\eta\wedge\eta_I+\omega^T.
\end{equation}
The tensors $g^T$, $I^T:=I|_{\mathcal H_I}$, and $\omega^T$ are
basic and define a transverse Hermitian structure. The torsion $H$ has the following expression
\begin{equation} \label{eqn:torsionH}
   H= \eta \wedge d\eta + \eta_I \wedge d\eta_I+ H^T 
\end{equation}
with $H^T$ being the Bismut torsion of the transverse Hermitian structure. 
\begin{prop}[cf. {\cite[Proposition~2.10]{ABLS}}]
\label{prop:compact-BHE-transverse-equations}
Given $(M,I,g)$ a compact (non K\"ahler) BHE manifold, the
transverse geometry satisfies
\begin{align}
    &\theta_I^T=df,
    \label{eq:BHE-transverse-Lee}\\
    &\operatorname{tr}_{\omega^T}d \eta=0,
    \\
    &\operatorname{tr}_{\omega^T}d \eta_I=-4,
    \label{eq:BHE-curvature-traces}\\
    &dd_{I^T}^c\omega^T
    =d \eta \wedge d \eta
      +d \eta_I \wedge d \eta_I .
    \label{eq:BHE-anomaly-equation}
\end{align}
Moreover,
\begin{equation}\label{eq:BHE-transverse-Ricci}
    d\eta_I=-\tfrac12\rho^{B,I^T}.
\end{equation}
\end{prop}

\begin{proof}
The identities \eqref{eq:BHE-transverse-Lee}--
\eqref{eq:BHE-anomaly-equation} follow from
\cite[Proposition~2.10]{ABLS}, with the constants adjusted
to our definition \eqref{eq:canonical-BHE-vector-field}
and unit-length normalization
\eqref{eq:BHE-unit-normalization}. It remains to verify
\eqref{eq:BHE-transverse-Ricci}. The local principal
$\R^2$-bundle description and \cite[Proposition~4]{GGP} give
\begin{equation}\label{eq:BHE-Ricci-reduction}
    \rho^{B,I}
    =\rho^{B,I^T}
    -\tfrac12\bigl(\operatorname{tr}_{\omega^T}d\eta\bigr)d\eta
    -\tfrac12\bigl(\operatorname{tr}_{\omega^T}d\eta_I\bigr)d\eta_I.
\end{equation}
Since $\rho^{B,I}=0$, the two trace identities in
\eqref{eq:BHE-curvature-traces} imply $0=\rho^{B,I^T}+2d\eta_I,$
which is precisely \eqref{eq:BHE-transverse-Ricci}.
\end{proof}

Equation \eqref{eq:BHE-transverse-Lee} implies that the transverse
Hermitian structure is conformally balanced. In real dimension six,
this implies that $\omega_{K}^T:=e^{-f}\omega^T$ is a transverse
K\"ahler form. In complex dimension two the Bismut Ricci form is
conformally invariant. Hence \eqref{eq:BHE-curvature-traces} and
\eqref{eq:BHE-transverse-Ricci} yield the following dimensional
reduction and converse construction. Notice that the numerical
constants differ from those in \cite[Proposition~2.18]{ABLS}, since
our unit vector field is one half of the vector field used there.

\begin{thm}[cf. {\cite[Proposition~2.18]{ABLS}}]
\label{thm:BHE-six-to-four}
Let $(M^6,I,g)$ be a compact non-K\"ahler BHE manifold, and let
$(M^T,g^T,I^T)$ be any local leaf space of $\mathcal F_I$, endowed
with the induced transverse Hermitian structure. Then $g_K:=e^{-f}g^T$
is K\"ahler. Denote its K\"ahler form, Ricci form, and scalar
curvature by $\omega_K$, $\rho_K$, and $R_K$, respectively.
Then:
\begin{enumerate}
    \item $d\eta_I=-\tfrac12 \rho_K$ and 
    \begin{equation}\label{eq:BHE-transverse-scalar-curvature}
        R_K=8e^f>0.
    \end{equation}

    \item
    The form $\alpha:=d\eta$ is a closed primitive real $(1,1)$-form
    satisfying
    \begin{equation}\label{eq:BHE-four-dimensional-equation}
        \tfrac12\bigl(dd_{I^T}^cR_K\bigr)\wedge\omega_K
        =4\alpha \wedge\alpha
         +\rho_K \wedge\rho_K.
    \end{equation}
\end{enumerate}
Conversely, let $(S,I^T,g_K,\omega_K)$ be a K\"ahler
orbifold surface with strictly positive scalar curvature,
and let $\alpha$ be a closed primitive real $(1,1)$-form
satisfying \eqref{eq:BHE-four-dimensional-equation}.
Assume that
\[
\operatorname{span}_{\mathbb R}
\bigl([\alpha],c_1^{\mathrm{orb}}(S,I^T)\bigr)
\subset H^2_{\mathrm{orb}}(S,\mathbb R)
\]
admits a basis consisting of integral orbifold cohomology
classes. Let $p:M\to S$ be a principal
$\mathbb T^2$-orbibundle with connection
one-forms $\eta,\eta_I$ satisfying
\[
d\eta=p^*\alpha,
\qquad
d\eta_I=-\tfrac12p^*\rho_K.
\]
Define
\begin{equation}\label{eq:BHE-four-to-six-data}
g^T:=\frac{R_K}{8}g_K,
\qquad
f:=\log\!\left(\frac{R_K}{8}\right).
\end{equation}
These data define a BHE structure on $M$ by
\[
g:=\eta^2+\eta_I^2+e^fg_K,
\qquad
\omega_I:=\eta\wedge\eta_I+e^f\omega_K.
\]
 Here $V_I$ and $W_I$ are the vertical vector
fields dual to $\eta$ and $\eta_I$, respectively, and
$I$ is induced by $I^T$ on the horizontal distribution
and satisfies $I(V_I)=W_I$ on the vertical distribution.
The Lee form is
\[
\theta_I=2\eta+df.
\]
If $S$ is compact and the total space $M$ is smooth,
the resulting structure is a compact non-K\"ahler
BHE threefold.
\end{thm}
\subsection{Generalized K\"ahler BHE structures}
We introduce the class of structures that will be studied throughout this paper.

\begin{defn}\label{def:GKBHE}
A generalized K\"ahler structure $(I,J,g)$ is called
\emph{generalized K\"ahler Bismut--Hermitian--Einstein}, or
\emph{GKBHE}, if $ \rho^{B,I}=0$ and $\rho^{B,J}=0$.
\end{defn}

Equivalently, a GKBHE structure consists of two BHE Hermitian
structures whose Bismut torsions have opposite signs. Indeed, setting $
H:=-d_I^c\omega_I=d_J^c\omega_J,$
the corresponding Bismut connections are
\begin{equation} \label{eq:Bismut_connections}
   \nabla^{B,I}=\nabla+\tfrac12H,
    \qquad
    \nabla^{B,J}=\nabla-\tfrac12H. 
\end{equation}

\begin{rmk}
It is natural to ask whether, within the generalized K\"ahler class,
the vanishing of only one of the two Bismut Ricci forms forces the
other to vanish. We do not know whether this implication holds, even
when the underlying manifold is compact.
\end{rmk}

Suppose now that $M$ is compact. Since both $(I,g)$ and $(J,g)$ are
BHE, they are  gradient pluriclosed solitons. Let $f$ be the normalized soliton potential associated with $(I,g)$. The system
\eqref{eq:-pluriclosed-soliton} is invariant under
$H\mapsto-H$. Since the Bismut torsion of $(J,g)$ is $-H$, the same
function $f$ is also the normalized soliton potential for $(J,g)$. Consequently,
$(I,J,g,f)$ is a generalized K\"ahler--Ricci soliton.

We define the associated canonical vector fields by
\[
    V_I:=\tfrac12\bigl(\theta_I^\sharp-\grad f\bigr),
    \qquad
    V_J:=\tfrac12\bigl(\theta_J^\sharp-\grad f\bigr).
\]
The structure theory of compact BHE manifolds 
\cite[Proposition~2.6]{ABLS}, together with the general symmetry
identities for generalized K\"ahler--Ricci solitons
\cite[Proposition~2.6]{ABSU}, yields the following collection of
a priori symmetries.
\begin{lem} \label{lem:a_priori_symmetries}
Let $(M,I,J,g)$ be a connected compact GKBHE manifold. The following a priori symmetries hold
\begin{enumerate}
    \item $V_I$ is $\nabla^{B,I}$-parallel, Killing and $I$-holomorphic; \label{item:V_I}
    \item $V_J$ is $\nabla^{B,J}$-parallel, Killing and $J$-holomorphic;\label{item:V_J}
    \item $IV_I$ is $\nabla^{B,I}$-parallel, Killing, $I$-holomorphic and $J$-holomorphic;\label{item:IV_I}
    \item $JV_J$ is $\nabla^{B,J}$-parallel, Killing, $I$-holomorphic and $J$-holomorphic;\label{item:JV_J}
    \item $\mathcal{L}_{V_{I}}f=\mathcal{L}_{IV_{I}}f=\mathcal{L}_{V_{J}}f=\mathcal{L}_{JV_{J}}f=0$.\label{item:basic}
\end{enumerate}
\end{lem}
We now establish further a priori-symmetries.
\begin{lem} \label{lem:commuting}
Let $(M,I,J,g)$ be a connected compact GKBHE manifold. Then, for any $X,Y \in \{V_I, IV_I, V_J,JV_J\}$, $[X,Y]=0$.
\end{lem}
\begin{proof}
For $(X,Y)\in \{(V_I,IV_I), (V_J,JV_J)\}$, the result follows using that $V_I$ and $V_J$ are $I$-holomorphic and $J$-holomorphic respectively (item \ref{item:V_I}, \ref{item:V_J} in Lemma \ref{lem:a_priori_symmetries}). Now, let $X \in \{V_I,IV_I\}$, and $Y \in \{V_J,JV_J\}$. Since $X$ is $\nabla^{B,I}$-parallel, 
\[
\nabla_Y X= -\tfrac12 H(Y,X). 
\]
Similarly, since $Y$ is $\nabla^{B,J}$-parallel, 
\[
\nabla_X Y= \tfrac12 H(X,Y). 
\]
Therefore
\[
[X,Y]=\nabla_X Y-\nabla_Y X=\tfrac12 H(X,Y)+\tfrac12 H(Y,X)=0. 
\]
\end{proof}
\begin{lem} \label{lem:IJ_holomorphic}
Let $(M,I,J,g)$ be a connected compact GKBHE manifold. Then $V_I$ is $J$-holomorphic and $V_J$ is $I$-holomorphic.    
\end{lem}
\begin{proof}
This is essentially contained in \cite[Proposition 4.2]{GS}. Indeed since $\rho^{B,I}=0$ and $V_J$ is $J$-holomorphic (item \ref{item:V_J} in Lemma \ref{lem:a_priori_symmetries}),
\[
0= \mathcal{L}_{\theta_J^\sharp-\theta_I^\sharp} J=2 \mathcal{L}_{V_J-V_I}J= -2 \mathcal{L}_{V_I}J.
\]
Exchanging the role of $I$ and $J$ in \cite[Proposition 4.2]{GS}, the same proof gives $ \mathcal{L}_{V_J}I=0$. 
\end{proof}
\begin{lem} \label{lem:norm}
Let $(M,I,J,g)$ be a connected compact GKBHE manifold. Then $|V_I|^2=|V_J|^2$.    
\end{lem}
\begin{proof}
Let $K \in \{I,J\}$. Then $\theta_K^\sharp=2 V_K + \grad(f)$. Using item \ref{item:basic} in Lemma \ref{lem:a_priori_symmetries}, $|\theta_K|^2=4 |V_K|^2+|\grad(f)|^2$. Since $V_I$ and $V_J$ are Killing (items \ref{item:V_I}, \ref{item:V_J} in Lemma \ref{lem:a_priori_symmetries}), 
$0=\mathcal{L}_{V_I-V_J}g=\tfrac 12 \mathcal{L}_{\theta_I^\sharp-\theta_J^\sharp}g$. Tracing this equation with the metric, we get $d^*\theta_I=d^*\theta_J$. We apply \eqref{eqn:torsion_trace} to both $(I,g,H)$ and $(J,g,-H)$. This gives
\begin{align*}
    &2 |\theta_I|^2 + 2 d^*\theta_I-\tfrac13 |H|^2=0, \\
    &2 |\theta_J|^2 + 2 d^*\theta_J-\tfrac13 |H|^2=0.
\end{align*}
Using that $d^*\theta_I=d^*\theta_J$, the difference of the two equations above gives $|\theta_I|^2=|\theta_J|^2$, which in turn imply the claimed identity. 
\end{proof}
If $V_I=0$, then $(I,g)$ is K\"ahler; consequently $H=0$, and $(J,g)$ is
K\"ahler as well. Thus, outside the Ricci-flat K\"ahler case, the
common length in Lemma~\ref{lem:norm} is positive.
Henceforth we assume that when the GKBHE structure is non-K\"ahler,
\begin{equation}\label{eq:GKBHE-unit-normalization}
    |V_I|^2=|V_J|^2=1,
\end{equation}
up to rescaling. 

We use the BHE splitting associated with $(I,g)$. As in Subsection \ref{subsec:BHE-structures}, let $\mathcal F_I:=\operatorname{span}\{V_I,IV_I\}$ and $\mathcal H:=\mathcal F_I^\perp$.
Then
\begin{equation}\label{eqn:splitting}
    TM=\mathcal F_I\oplus^\perp\mathcal H.
\end{equation}
The Hermitian structure $(g,\omega_I)$ decomposes accordingly. Set $\eta:=V_I^\flat$ and $\eta_I:= (IV_I)^\flat$.
Define $Z_1:=V_J$, and $Z_2:=JV_J$.
With respect to \eqref{eqn:splitting}, write
\begin{equation}\label{eqn:Z_1Z_2}
    Z_i=a_iV_I+b_iIV_I+X_i,
    \qquad i=1,2,
\end{equation}
where $X_i\in\Gamma(\mathcal H)$ and $a_i=g(Z_i,V_I)$ and $b_i=g(Z_i,IV_I).$
\begin{prop}\label{prop:GKBHE-horizontal-symmetries}
Let $(M^{2n},I,J,g)$, be a connected compact GKBHE manifold. For all
$i,j\in\{1,2\}$, the following hold.
\begin{enumerate}[label=\textup{(\roman*)}, ref=(\roman*)]
    \item\label{item:GKBHE-coefficients-basic}
    The functions $a_i$ and $b_i$ are basic.

    \item\label{item:GKBHE-coefficients-invariant}
    One has $Z_i(a_j)=Z_i(b_j)=X_i(a_j)=X_i(b_j)=0.$

    \item\label{item:GKBHE-horizontal-projectable}
    The vector fields $X_i$ satisfy $[X_i,V_I]=[X_i,IV_I]=0.$

    \item\label{item:GKBHE-horizontal-commuting}
    The vector fields $X_1$ and $X_2$ commute.

    \item\label{item:GKBHE-horizontal-f-invariant}
    One has $X_i(f)=0$.

    \item\label{item:GKBHE-transverse-automorphisms}
    The vector fields $X_i$ project to real holomorphic Killing vector
    fields of the transverse Hermitian structure
    $(g^T,\omega_I^T,I^T)$. If $n \ge 3$, they also preserve the transverse conformally 
    balanced metric $(e^{-f/(n-2)}g^T, e^{-f/(n-2)}\omega_I^T)$.
    \item\label{item:GKBHE-curvature-potentials}
    The following identities hold:
    \begin{equation} \label{eqn:diff_curvatures}
        da_i=-\iota_{X_i}d\eta,
        \qquad
        db_i=-\iota_{X_i}d\eta_I.
    \end{equation}

    \item\label{item:GKBHE-curvature-orthogonality}
    One has $d\eta(X_1,X_2)=d\eta_I(X_1,X_2)=0.$
\end{enumerate}
\end{prop}

\begin{proof} 
\begin{enumerate}
\renewcommand{\theenumi}{\roman{enumi}}
\renewcommand{\labelenumi}{\textup{(\theenumi)}}
\item
    Let $W\in\{V_I,IV_I\}$. Since $W$ is Killing (items \ref{item:V_I}, \ref{item:IV_I} in Lemma \ref{lem:a_priori_symmetries}) and commutes with
    $Z_j,V_I$, and $IV_I$ (Lemma \ref{lem:commuting}), we have
    \[
        W(a_j)
        =(\mathcal L_Wg)(Z_j,V_I)
          +g([W,Z_j],V_I)+g(Z_j,[W,V_I])=0.
    \]
    The same computation gives $W(b_j)=0$.

    \item
    Applying the same argument with $W=Z_i$ gives
    $Z_i(a_j)=Z_i(b_j)=0$. The decomposition
    \eqref{eqn:Z_1Z_2} and item~\ref{item:GKBHE-coefficients-basic}
    then give $X_i(a_j)=X_i(b_j)=0$.

    \item
    For $W\in\{V_I,IV_I\}$, pairwise commutativity (Lemma \ref{lem:commuting}) and
    item~\ref{item:GKBHE-coefficients-basic} give
    \[
        0=[Z_i,W]=[a_iV_I+b_iIV_I+X_i,W]=[X_i,W].
    \]

    \item
    Expanding $[Z_1,Z_2]=0$ (Lemma \ref{lem:commuting}) and using
    items~\ref{item:GKBHE-coefficients-basic}--
    \ref{item:GKBHE-horizontal-projectable}, all terms involving the
    vertical fields vanish. Hence $[X_1,X_2]=0$.

    \item
    This follows immediately from \eqref{eqn:Z_1Z_2}, since
    $Z_i(f)=V_I(f)=IV_I(f)=0$ (item \ref{item:basic} in Lemma \ref{lem:a_priori_symmetries}).

    \item
    Each $Z_i$ preserves $g,I,\eta$, and $\eta_I$; hence it preserves
    $g^T$ and $\omega_I^T$ (items \ref{item:V_J}, \ref{item:JV_J} in Lemma \ref{lem:a_priori_symmetries}, Lemma \ref{lem:IJ_holomorphic}
    and Lemma \ref{lem:commuting}). Since these transverse tensors are basic,
    the vertical part $a_iV_I+b_iIV_I$ has zero Lie derivative on
    $g^T$ and $\omega_I^T$. Therefore
    \[
        \mathcal L_{X_i}g^T=0,
        \qquad
        \mathcal L_{X_i}\omega_I^T=0.
    \]
    Thus $X_i$ is transversely Killing and real holomorphic. The same
    conclusion holds for the transverse conformally balanced metric by
    item~\ref{item:GKBHE-horizontal-f-invariant}.

    \item
    Since $Z_i$ is Killing and commutes with $V_I$ and $IV_I$ (items \ref{item:V_J} and \ref{item:JV_J} in Lemma \ref{lem:a_priori_symmetries} and Lemma \ref{lem:commuting}), it
    preserves both $\eta$ and $\eta_I$. Cartan's formula and the
    basicness of $d\eta,d\eta_I$ (see Subsection \ref{subsec:BHE-structures}) therefore give
    \[
        0=\mathcal L_{Z_i}\eta
          =\iota_{X_i}d\eta+da_i,
        \qquad
        0=\mathcal L_{Z_i}\eta_I
          =\iota_{X_i}d\eta_I+db_i.
    \]

    \item
    By items~\ref{item:GKBHE-coefficients-invariant} and
    \ref{item:GKBHE-curvature-potentials},
    \[
        d\eta(X_1,X_2)=-X_2(a_1)=0,
        \qquad
        d\eta_I(X_1,X_2)=-X_2(b_1)=0.
    \]
\end{enumerate}
\end{proof}

\begin{rmk} \label{rmk:changing_J_sign}
Let $(I,J,g)$ be a GKBHE structure. Then $(I,-J,g)$ is also GKBHE. As explained in Section~\ref{subs:GK}, the triple $(I,-J,g)$ remains generalized Kähler. It therefore remains only to verify that its Bismut Ricci form vanishes. Replacing $J$ by $-J$ leaves the Bismut connection unchanged, so $\mathrm{R}^{B,-J}=\mathrm{R}^{B,J}.$ Since the definition of the Bismut Ricci form is linear in the complex structure (see equation~\eqref{eqn:rho_B}), we obtain $\rho^{B,-J}=-\rho^{B,J}=0.$
\end{rmk}
\subsection{The horizontal rank $r$}
In this subsection, we introduce the \emph{horizontal rank} of a compact GKBHE structure. This notion should not be confused with the rank of a generalized K\"ahler--Ricci soliton introduced in \cite{ABSU}.

\begin{defn}
Let $(M,I,J,g)$ be a connected compact GKBHE manifold. Set $Z_1:=V_J$ and $Z_2:=JV_J$,
and let \(X_i\) denote the \(\mathcal H\)-component of \(Z_i\) with respect to the orthogonal splitting \eqref{eqn:splitting}. Thus $Z_i=a_iV_I+b_iIV_I+X_i$, for $i=1,2$. We define the \emph{horizontal rank} of the GKBHE structure by
\begin{equation}\label{eqn}
r
:=\max_{p\in M}
\dim_{\mathbb R}\big(\operatorname{span}
\{X_1(p),X_2(p)\}\big)
\in\{0,1,2\}.
\end{equation}
\end{defn}
\begin{rmk}
The definition above applies to non-K\"ahler GKBHE structures.
In the K\"ahler case, the canonical vector fields $V_I$ and $V_J$
vanish, and we set $r=0$ by convention.
\end{rmk}
\begin{rmk}
The GKBHE condition has a natural hypercomplex counterpart in generalized hyperK\"ahler geometry \cite{Bredthauer}. A generalized hyperK\"ahler structure consists of two strong HKT structures $(g,I,J,K)$ and $(g,\hat I,\hat J,\hat K)$ sharing the same metric and having opposite Bismut torsion forms. Since each strong HKT structure induces a BHE structure for every complex structure in its hypercomplex triple, choosing one complex structure from each triple yields a GKBHE structure. This suggests extending the present analysis to the hypercomplex setting. Indeed, in the compact, non-hyperK\"ahler case, the two strong HKT structures determine canonical quadruples of Killing vector fields
\[
V_I,\ IV_I,\ JV_I,\ KV_I,
\qquad
V_{\hat I},\ \hat I V_{\hat I},\ \hat J V_{\hat I},\ \hat K V_{\hat I},
\]
parallel with respect to their respective Bismut connections \cite{BFGV,Papadopoulos}. If the manifold is also simply connected and eight-dimensional, each quadruple spans a Lie algebra isomorphic to $\mathfrak{u}(1)\oplus\mathfrak{su}(2)$ \cite{BFGV}.
\end{rmk}
We conclude this section by considering the case of horizontal rank zero.

\begin{prop} \label{prop:rank0}
Let $(M,I,J,g)$ be a connected compact GKBHE manifold with \(r=0\). Then \((I,g)\) and $(J,g)$ are K\"ahler.
\end{prop}
\begin{proof}
Since \(r=0\), we have \(X_1=X_2=0\), and hence $\mathcal F_J
:=\operatorname{span}\{V_J,JV_J\}=\mathcal F_I.$
In particular,
\begin{equation}\label{eqn:Z_2}
Z_2=a_2V_I+b_2IV_I.
\end{equation}

By item \ref{item:GKBHE-coefficients-basic} in Proposition~\ref{prop:GKBHE-horizontal-symmetries},
the functions \(a_2\) and \(b_2\) are basic. We claim that they are in fact constant. Indeed, let \(Y\in\mathcal H\). Since \(Z_2,V_I\), and \(IV_I\) are Killing (items \ref{item:V_I}, \ref{item:IV_I}, \ref{item:JV_J} in Lemma \ref{lem:a_priori_symmetries}) and Lemma \ref{lem:commuting} holds,
\[
0=(\mathcal L_{Z_2}g)(Y,V_I)=Y(a_2),
\qquad
0=(\mathcal L_{Z_2}g)(Y,IV_I)=Y(b_2).
\]
Thus \(da_2\) and \(db_2\) vanish on \(\mathcal H\), and hence, using that \(a_2\) and \(b_2\) are basic, they are indeed constant.

Now \(Z_2=JV_J\) is \(\nabla^{B,J}\)-parallel by item \ref{item:JV_J} in Lemma~\ref{lem:a_priori_symmetries}. Since \(a_2,b_2\) are constant and \(V_I,IV_I\) are \(\nabla^{B,I}\)-parallel (items \ref{item:V_I}, \ref{item:IV_I} in Lemma \ref{lem:a_priori_symmetries}), Equation \eqref{eqn:Z_2} also implies that \(Z_2\) is \(\nabla^{B,I}\)-parallel. The generalized K\"ahler condition then gives that \(Z_2\) is Levi-Civita parallel. This implies that $Z_2^\flat$ is closed.

On the other hand, applying \eqref{eq:BHE-curvature-traces} to the BHE structure \((J,g)\), whose canonical vertical one-form associated with \(JV_J\) is \(Z_2^\flat\), gives $\operatorname{tr}_{\omega_J^T}dZ_2^\flat=-4$, contradicting \(dZ_2^\flat=0\). Then $V_J=JV_J=0$ and $(J,g)$ is K\"ahler. The generalized K\"ahler condition gives that $(I,g)$ is also K\"ahler, concluding the proof.
\end{proof}

\subsection{The horizontal rank $1$}

In general, the horizontal distribution $\mathcal H$ associated with $(I,g)$ need not be $J$-invariant. Equivalently, $J$ need not preserve the splitting $TM=\mathcal F_I\oplus\mathcal H$, but may mix its vertical and horizontal components. Thus, $J$ is not necessarily adapted to the local principal $\mathbb R^2$-bundle description determined by the canonical foliation $\mathcal F_I$ (cf. Subsection \ref{subsec:BHE-structures}). In the rank-zero case, $J$ is adapted to the above splitting and, as shown in Proposition \ref{prop:rank0}, the structure is necessarily K\"ahler. We therefore focus on the rank-one case. 
In this subsection, we collect some general results about the case $r=1$. Observe that $r=1$ implies that the structure is non-K\"ahler.

First, we prove that any compact non-K\"ahler GKBHE manifold with $r=1$ satisfies $V_I=\pm V_J$. 
\begin{rmk}\label{rmk:real_analiticity}
Standard analytic elliptic regularity applied to the  pluriclosed soliton system \eqref{eq:-pluriclosed-soliton} shows that every compact BHE metric is real analytic in suitable local coordinates. Since every Killing vector field of a real-analytic Riemannian metric is itself real analytic, the same holds for $V_I,IV_I,V_J$, and $JV_J$. Consequently, the horizontal projections $X_1$ and $X_2$ are real analytic as well.
\end{rmk}
\begin{lem} \label{lem:linear_combination}
Let $(M,I,J,g)$ be a connected compact GKBHE manifold with $r=1$. Then there exists $(c_1,c_2) \in \R^2 \setminus \{(0,0)\}$ such that $c_1 X_1+c_2 X_2=0$.
\end{lem}
\begin{proof}
If $X_2\equiv0$, the conclusion holds with $(c_1,c_2)=(0,1)$. Suppose therefore that $X_2$ does not vanish identically, and choose $p\in M$ such that $X_2(p)\neq0$. Let $\mathcal U_p$ be a connected neighborhood of $p$ on which $X_2$ is nowhere vanishing. Since $r=1$, there exists $\lambda\in C^\infty(\mathcal U_p)$ such that $X_1=\lambda X_2$ on $\mathcal U_p$. By item \ref{item:GKBHE-horizontal-projectable} in Proposition \ref{prop:GKBHE-horizontal-symmetries}, both $X_1$ and $X_2$ commute with $V_I$ and $IV_I$. Applying these Lie derivatives to $X_1=\lambda X_2$ shows that $V_I(\lambda)=IV_I(\lambda)=0$, so $\lambda$ is basic. Moreover, by item \ref{item:GKBHE-transverse-automorphisms} in Proposition \ref{prop:GKBHE-horizontal-symmetries}, $X_1$ and $X_2$ are transversely Killing. Since a non-vanishing Killing vector field and a smooth multiple of it can both be Killing only if the multiplying function is constant, $\lambda$ is transversely constant. Together with its basicness, this shows that $\lambda\equiv\lambda_0$ on $\mathcal U_p$ for some $\lambda_0\in\R$.

By Remark \ref{rmk:real_analiticity}, the vector field $X_1-\lambda_0X_2$ is real analytic. Since it vanishes on the nonempty open set $\mathcal U_p$, it vanishes identically on $M$. The conclusion follows by taking $(c_1,c_2)=(1,-\lambda_0)$.
\end{proof}
Using the preceding lemma, we now prove the desired result.
\begin{prop}\label{prop:linear_dependence}
Let $(M,I,J,g)$ be a connected compact GKBHE manifold with $r=1$. Then $V_I=c V_J$, $c \in \{-1,1\}$.
\end{prop}
\begin{proof}
Let $(c_1,c_2)\in\R^2\setminus{(0,0)}$ be as in Lemma \ref{lem:linear_combination}, and consider the Killing vector field $Z:=c_1Z_1+c_2Z_2$. Since $c_1X_1+c_2X_2=0$, we have
\begin{equation}\label{eqn:Z}
Z=(c_1a_1+c_2a_2)V_I+(c_1b_1+c_2b_2)IV_I.
\end{equation}
Set $A:=c_1a_1+c_2a_2$ and $B:=c_1b_1+c_2b_2$. By item \ref{item:GKBHE-coefficients-basic} in Proposition \ref{prop:GKBHE-horizontal-symmetries}, both $A$ and $B$ are basic. Moreover, for every $Y\in\mathcal H$, the fact that $Z,V_I$, and $IV_I$ are Killing gives $(\mathcal L_Zg)(Y,V_I)=Y(A)=0$ and $(\mathcal L_Zg)(Y,IV_I)=Y(B)=0$. Thus $dA$ and $dB$ vanish on $\mathcal H$, while their basicness implies that they also vanish on $\mathcal F_I$. Hence $A$ and $B$ are constant.

Since $V_I$ and $IV_I$ are $\nabla^{B,I}$-parallel, Equation \eqref{eqn:Z} shows that $Z$ is $\nabla^{B,I}$-parallel. On the other hand, $Z=c_1V_J+c_2JV_J$ is $\nabla^{B,J}$-parallel. The generalized K\"ahler condition therefore implies that $\nabla Z=0$. In particular, $dZ^\flat=0$.

Taking the trace of $0=dZ^\flat=A \,d \eta +B \, d \eta_I$ with respect to $\omega_I^T$ and using \eqref{eq:BHE-curvature-traces}, we obtain $0=-4B$. Thus $B=0$ and $Z=AV_I$. Applying the corresponding trace identity to the BHE structure $(J,g)$ and using $Z=c_1V_J+c_2JV_J$, we similarly obtain $0=-4c_2$. Hence $c_2=0$, and therefore $c_1\neq0$.

Since $(I,g)$ is non-K\"ahler and $|V_I|=|V_J|=1$ by Lemma \ref{lem:norm}, both vector fields are nonzero. Thus $Z=c_1V_J=AV_I$ implies that $A\neq0$. Taking norms and using again $|V_I|=|V_J|=1$, we obtain $|A|=|c_1|$. Consequently, $V_J=\frac{A}{c_1}V_I=\pm V_I$, as claimed.
\end{proof}

\begin{cor} \label{cor:V_Iparallel}
Let $(M,I,J,g)$ be a connected compact GKBHE manifold with $r=1$. Then $\nabla V_I=0$, and $d\eta=0$. Furthermore $\iota_{V_{I}}H=0$.
\end{cor}
\begin{proof}
The proof of this corollary is implicit in the proof of Proposition \ref{prop:linear_dependence}. The last claim follows by Equation \eqref{eqn:torsionH}. 
\end{proof}

\begin{rmk}
The horizontal rank-one case is closely related to Sasaki
with torsion geometry in the sense of \cite{BFFG}, the
odd-dimensional counterpart of K\"ahler with torsion geometry.
By Corollary~\ref{cor:V_Iparallel}, $V_I$ and $V_J$ are
Levi-Civita parallel, so $(M,g)$ locally splits isometrically
as $\mathbb R\times N$. The two Hermitian structures induce
on $N$ a pair of strong Sasaki with torsion structures with
opposite torsions and vanishing Ricci forms for their
respective characteristic connections.
This description suggests exploring  an odd-dimensional counterpart
of generalized K\"ahler geometry.
\end{rmk}

We show that GKBHE structures of horizontal rank one occur in naturally associated pairs corresponding to the two values \(c=\pm1\). More precisely, starting from a structure \((I,J,g)\) satisfying \(V_J=cV_I\), we construct a second GKBHE structure \((I,\widehat J,g)\) on the same manifold, with the same metric and first complex structure, for which \(V_{\widehat J}=-cV_I\). This correspondence is defined intrinsically by an orthogonal reflection and is involutive: applying it twice recovers the original structure. Consequently, the existence of a structure with either value of \(c\) implies the existence of one with the opposite value.
\begin{prop}\label{prop:GK-parallel-reflection}
Let $(M,g,J)$ be a connected Hermitian manifold with Bismut
torsion $H$. Suppose that there exists a unit vector field $V$
satisfying $\nabla V=0$ and  $\iota_VH=0.$ Denote $\eta:=V^\flat$ and set $$
Q:=\operatorname{Id}-2\eta\otimes V, \quad
\widehat J:=-QJQ.
$$
Then $(g,\widehat J)$ is Hermitian, $H_{\widehat J}=H,$ 
$
\theta_{\widehat J}=\theta_J(Q\cdot),
$ and $
\rho^{B,\widehat J}=-\rho^{B,J}.$
In particular, the CYT and BHE conditions are preserved.
\end{prop}

\begin{proof}
Using that $Q$ is orthogonal and $Q^2=\operatorname{Id}$, the endomorphism
$\widehat J$ is orthogonal and satisfies
$\widehat J^2=-\operatorname{Id}$. Therefore $(\widehat J, g)$ is almost Hermitian. 
Write $
D:=\nabla^{B,J}=\nabla+\tfrac12g^{-1}H.$
The hypotheses give $DV=0$. Since $D$ is metric, it follows that
$DQ=0$. This, together with $DJ=0$,
implies $D\widehat J=0$. Therefore $D$ is a metric connection with skew torsion that preserves the almost complex structure $\widehat J$. By \cite[Theorem~10.1 and its proof]{FI},
the tensor
$N_{\widehat J}(X,Y,Z):=g(N_{\widehat J}(X,Y),Z)$ is a three-form and
satisfies
\begin{equation} \label{eqn:Ninj}
\begin{split}
N_{\widehat J}(X,Y,Z)
&=H(X,Y,Z)-H(\widehat JX,\widehat JY,Z)-H(\widehat JX,Y,\widehat JZ)
-H(X,\widehat JY,\widehat JZ).
\end{split}
\end{equation}
Moreover, $Q$ reverses $V$ and fixes $V^\perp$, so
$\iota_VH=0$ implies
\begin{equation}\label{eqn:Q_invariance}
H(QX,QY,QZ)=H(X,Y,Z).
\end{equation}
Substituting $\widehat J=-QJQ$ in Equation \eqref{eqn:Ninj} and using
\eqref{eqn:Q_invariance}, we obtain
\[
\begin{aligned}
N_{\widehat J}(X,Y,Z)
&=H(QX,QY,QZ)-H(JQX,JQY,QZ)-H(JQX,QY,JQZ)-H(QX,JQY,JQZ)\\
&=N_J(QX,QY,QZ)=0.
\end{aligned}
\]
Thus $\widehat J$ is integrable. By the uniqueness of the Hermitian
connection with totally skew-symmetric torsion, $D$ is the Bismut
connection of $(\widehat J, g)$ and hence $H_{\widehat J}=H$. 

Let $\mathrm{R}^D$ denote the curvature of $D$. Since $DQ=0$, the endomorphism
$R^D(X,Y)$ commutes with $Q$. Therefore we may compute the Bismut Ricci form of $(\widehat J, g)$ as follows:
\[
\begin{aligned}
\rho^{D,\hat J}(X,Y)
&=\tfrac12\operatorname{tr}\bigl(\mathrm{R}^D(X,Y)\widehat J\bigr)=-\tfrac12\operatorname{tr}\bigl(\mathrm{R}^D(X,Y)QJQ\bigr)\\
&=-\tfrac12\operatorname{tr}\bigl(Q\mathrm{R}^D(X,Y)JQ\bigr)=-\tfrac12\operatorname{tr}\bigl(\mathrm{R}^D(X,Y)J\bigr)\\
&=-\tfrac12\operatorname{tr}\bigl(\mathrm{R}^{B,J}(X,Y)J\bigr)
=-\rho^{B,J}(X,Y),
\end{aligned}
\]
where we used the commutation with $Q$, the cyclicity of the trace,
$Q^2=\operatorname{Id}$ and the fact that $D=\nabla^{B,J}$. Hence the CYT condition is preserved.
Since the torsion is unchanged, its closedness is also preserved,
and consequently so is the BHE condition. 

Finally, let $\{e_i\}_{i=1}^{2n}$ be a local orthonormal frame.
Using the torsion formula for the Lee form in the orthonormal frame
$\{Qe_i\}$, together with \eqref{eqn:Q_invariance}, gives
\begin{align*}
\theta_{\widehat J}(X)
&=-\tfrac12\sum_i H(\widehat JX,Qe_i,\widehat JQe_i)\\
&=-\tfrac12\sum_i H(JQX,e_i,Je_i)
=\theta_J(QX).
\end{align*}
\end{proof}
\begin{cor}\label{cor:GKBHE-sign-involution}
Let $(M,I,J,g)$ be a connected compact GKBHE manifold with $r=1$. If $V_I=cV_J$ for some $c\in\{1,-1\}$, then $(I,\widehat J,g)$ is GKBHE and $V_{\widehat J}=-cV_I$.
\end{cor}
\begin{proof}
By Corollary~\ref{cor:V_Iparallel}, $V_I$ is Levi-Civita parallel,
and $\iota_{V_I}H=0$. Since $H_J=-H$,
Proposition~\ref{prop:GK-parallel-reflection} applies to $(J,g)$ and yields that $(I,\widehat J, g)$ is GKBHE.
The normalized soliton potential remains $f$, since $(g,\widehat J)$
has the same metric and Bismut torsion as $(J,g)$.
Since $V_J=cV_I$, we have $\theta_J=df+2c\eta$.
By Proposition \ref{prop:GK-parallel-reflection}, the basicness of $f$ gives $
\theta_{\widehat J}=df-2c\eta,$
and hence $
V_{\widehat J}
=\tfrac12\bigl(\theta_{\widehat J}^{\sharp}
-\operatorname{grad} f\bigr)
=-cV_I.
$
\end{proof}
\section{GKBHE $3$-folds: general theory in the horizontal rank $1$ case}
Starting from this section, we work in real dimension $6$.
Let $(M^6,I,J,g)$ be a connected compact GKBHE manifold with $r=1$; then $(I,g)$ is non-K\"ahler and the transverse geometry of the BHE structure $(I, g)$ is described in Theorem \ref{thm:BHE-six-to-four}. Our aim is to determine the additional constraints imposed on this local transverse geometry by the existence of a second complex structure $J$ such that $(J,g)$ is BHE and $(I,J,g)$ is generalized K\"ahler. We shall prove that, in this setting, the local transverse geometry is governed by a single nonnegative function $F$ such that $I^T(\grad_{g_K^T}F)$ is $I^T$-holomorphic and, on the dense open set where $F>0$, $\log F$ is a Ricci potential for the transverse K\"ahler metric.

By Proposition \ref{prop:linear_dependence}, we have $Z_1=cV_I$ for some $c\in\{\pm1\}$, while $Z_2=a_2V_I+b_2IV_I+X_2$. Since $Z_1$ and $Z_2$ are orthogonal, it follows that $a_2=0$. Henceforth, setting $b:=b_2$ and $X:=X_2$, we shall write
\begin{equation} \label{eqn:V_J/JV_J}
    V_J=cV_I,\qquad JV_J=bIV_I+X.
\end{equation}
Since this decomposition is orthogonal and $|JV_J|=|IV_I|=1$ (see Lemma \ref{lem:norm}), we have
\begin{equation} \label{eqn:norm_X}
 |X|^2=1-b^2.   
\end{equation}
In particular, for each $p \in M$, $b(p)\in[-1,1]$.
\subsection{The Poisson tensor $\sigma$}
In this subsection, we study the real Poisson tensor $\sigma=\tfrac12[I,J]g^{-1}$ (cf. Proposition \ref{prop:commutator-poisson}) associated with a connected, compact GKBHE manifold $(M^6,I,J,g)$ with $r=1$. Henceforth, we assume that $I$ and $J$ induce the same orientation on $M$. This entails no loss of generality, since otherwise we may replace $J$ by $-J$ (cf. Remark \ref{rmk:changing_J_sign}). This assumption will remain in force throughout the paper. 

We introduce the open set on which the Poisson tensor $\sigma$ does not vanish. Recall that, whenever $\sigma_p\neq0$,  $\dim_\R \mathrm{Im}[I,J]_p=4$ (cf. Subsection \ref{subs:GK}).

\begin{defn}
Let $(M^6,I,J,g)$ be a connected compact GKBHE manifold with $r=1$. We define
\begin{equation}\label{eqn:O}
\mathcal O:=\{p\in M\mid [I,J]_p\neq0\}=\{p\in M\mid \sigma_p\neq0\}.
\end{equation}
\end{defn}

Fix $p\in\mathcal O$ and set $\mathcal K_p:=\ker[I,J]_p,$ and $\mathcal {D}_p:=\operatorname{Im}[I,J]_p.$ Since $[I,J]_p$ is skew-adjoint, its kernel and image are orthogonal. Moreover, using that $\dim_\R \operatorname{Im}[I,J]_p=4$ (cf. Subection \ref{subs:GK}),
\begin{equation} \label{eqn:splitting2}
    T_pM=\mathcal K_p\oplus^\perp\mathcal {D}_p,
\qquad
\dim_{\mathbb R}\mathcal K_p=2,
\qquad
\dim_{\mathbb R}\mathcal {D}_p=4.
\end{equation}
The endomorphism $[I,J]_p$ is invertible on $\mathcal {D}_p$ and since $[I,J]$ anticommutes with $I$ and $J$, both $\mathcal K_p$ and $\mathcal {D}_p$ are invariant under $I_p$ and $J_p$.

The restrictions of $I_p$ and $J_p$ to $\mathcal {D}_p$ induce the same orientation. Indeed, suppose that they induce opposite orientations. With respect to the orientation induced by $I_p$, the restrictions of $\omega_I(p)$ and $\omega_J(p)$ to $\mathcal {D}_p$ would then be self-dual and anti-self-dual, respectively. However, it would follow that $I_p$ and $J_p$ commute on $\mathcal {D}_p$, contradicting the invertibility of $[I,J]_p$ there.

Finally, on the two-dimensional space $\mathcal K_p$, the orthogonal complex structures $I_p$ and $J_p$ differ by a sign, that is $I_p=\pm J_p$. Since they induce the same orientation on $T_pM$ and their restrictions to $\mathcal {D}_p$ induce the same orientation, their restrictions to $\mathcal K_p$ must also induce the same orientation. Therefore, $I_p|_{\mathcal K_p}=J_p|_{\mathcal K_p}.$

\begin{claim}  \label{claim:pi_pm}
Consider the Poisson tensors $\pi_\pm=\tfrac12(I\pm J)g^{-1}$ introduced in \eqref{eq:poisson-plus-minus}. For every $p\in\mathcal O$, $\dim_{\mathbb R}\operatorname{Im}\pi_{-}(p)=4,$ and $\dim_{\mathbb R}\operatorname{Im}\pi_{+}(p)=6.$
In particular, $I+J$ is invertible on $\mathcal O$.
\end{claim}

\begin{proof}
Since $g^{-1}$ is an isomorphism, $\dim_{\mathbb R}\operatorname{Im}\pi_{\pm}=\dim_{\mathbb R}\operatorname{Im}(I\pm J)$. We first observe that both $I_p-J_p$ and $I_p+J_p$ are invertible on $\mathcal D_p$. Indeed, let $\varepsilon\in \{\pm1 \}$ and suppose that $Y\in\mathcal D_p$ satisfies $(I_p+\varepsilon J_p)Y=0$. Then $I_pY=-\varepsilon J_pY$, and applying $I_p$ and $J_p$, respectively, gives $I_pJ_pY=\varepsilon Y=J_pI_pY$. Thus $[I,J]_pY=0$, and hence $Y\in\mathcal K_p\cap\mathcal D_p={0}$. Therefore, $(I_p+\varepsilon J_p)|_{\mathcal D_p}$ is invertible.

On the other hand, $I_p=J_p$ on $\mathcal K_p$, so $(I_p-J_p)|_{\mathcal K_p}=0$, whereas $(I_p+J_p)|_{\mathcal K_p}=2I_p$ is invertible. Consequently, $\ker(I_p-J_p)=\mathcal K_p$ and $\ker(I_p+J_p)={0}$. Since $\dim_{\mathbb R}\mathcal K_p=2$, the claim follows.
\end{proof}

We are now in a position to prove the main result of this subsection.

\begin{prop}\label{prop:vanishin_poisson}
Let $(M^6,I,J,g)$ be a connected compact non-K\"ahler GKBHE manifold with $r=1$. If $c=1$, i.e., $V_I=V_J$, then $[I,J]$ vanishes identically.
\end{prop}
\begin{proof}
Suppose, by contradiction, that $\mathcal O\neq\varnothing$. By Claim \ref{claim:pi_pm}, $\mathcal O\subset U_+$, where $U_+$ denotes the open set introduced in Subsection \ref{subs:GK} on which $I+J$ is invertible. We may therefore apply Lemma \ref{lem:divergence} on $\mathcal O$ to the normalized soliton potential $f$. Since $V_I=V_J$, it gives
$$
(I+J)V_I
=(IV_I+JV_J)
=\tfrac 12(I+J)(-\grad f+\grad\psi_+),
$$
where $\psi_+=\tfrac12\log\det(I+J)$. Since $I+J$ is invertible on $\mathcal O$, we obtain $
2 V_I=-\grad f+\grad\psi_+.$ Taking the inner product with $V_I$ yields $2|V_I|^2=-V_I(f)+V_I(\psi_+).$ By item \ref{item:basic} in Lemma \ref{lem:a_priori_symmetries}, $f$ is basic, and hence $V_I(f)=0$. Moreover, $V_I$ preserves both $I$ and $J$ by item \ref{item:V_I} in Lemma \ref{lem:a_priori_symmetries} and Lemma \ref{lem:IJ_holomorphic}. It therefore preserves $\det(I+J)$, so that $V_I(\psi_+)=0$. We conclude that $|V_I|^2=0$, contradicting $|V_I|=1$. Thus $\mathcal O=\varnothing$, and hence $[I,J]$ vanishes identically.
\end{proof}

We introduce the following definition 
\begin{defn}
Let $(M^6,I,J,g)$ be a connected compact GKBHE manifold with $r=1$. We define
\begin{equation}\label{eqn:U}
\mathcal U:=\{p\in M\mid |b(p)|<1\}=\{p\in M\mid X(p)\neq0\}.
\end{equation}
\end{defn}
The two descriptions in \eqref{eqn:U} agree because $|X|^2=1-b^2$ (cf. Equation \eqref{eqn:norm_X}). Moreover, since in the notation \eqref{eqn:Z_1Z_2} we have $X_1=0$ and $X_2=X$, the set $\mathcal U$ is precisely the locus where the pointwise horizontal rank is equal to one.
\begin{lem} \label{lem:density_U}
The open set $\mathcal U$ is dense in $M$.
\end{lem}
\begin{proof}
Suppose, by contradiction, that $\mathcal U$ is not dense. Then $M\setminus\overline{\mathcal U}$ contains a nonempty open subset $\mathcal V$. By definition, $X$ vanishes identically on $\mathcal V$. Since $X$ is real analytic by Remark \ref{rmk:real_analiticity}, it follows that $X$ vanishes identically on $M$. Hence, $r=0$, a contradiction.
\end{proof}
We have shown that if $V_I=V_J$, then the Poisson tensor $\sigma$ vanishes identically (Proposition \ref{prop:vanishin_poisson}). We now assume that $c=-1$ and study $\operatorname{Im}\sigma_p$ for $p \in \mathcal O\cap\mathcal U$. We first record an observation valid on all of $\mathcal O$. 

Let $p \in \mathcal{O}$. By Claim \ref{claim:pi_pm}, the Poisson tensors $\pi_\pm$ satisfy $\dim_{\mathbb R}\operatorname{Im}\pi_{-}=4,$ and $\dim_{\mathbb R}\operatorname{Im}\pi_{+}=6$ on $\mathcal O$. Therefore, the two generalized complex structures associated with $(I,J,g)$ through the Gualtieri map \cite[Equation~(1.18)]{Gualtieri2} have constant type on $\mathcal O$ \cite[Remark 1.13]{Gualtieri2}. The generalized Darboux theorem \cite[Theorem 4.3]{Gualtieri1} then implies that there exists an open neighborhood $\mathcal{V}_p \subset \mathcal{O}$ around $p$ on which both generalized complex structures are defined by nonvanishing $d_H$-closed pure spinors. Hence \cite[Proposition 4.3]{AFSU} applies locally on $\mathcal{V}_p$ and yields
$
\theta_I^\sharp-\theta_J^\sharp (p)\in\operatorname{Im}\sigma_p.
$
On the other hand, the identity $V_J=-V_I$ gives
$
\theta_I^\sharp-\theta_J^\sharp
=2(V_I-V_J)
=4V_I.
$
Consequently,
\begin{equation} \label{eqn:im_sigma}
  V_I(p)\in\operatorname{Im}\sigma_p  
\end{equation}
for every $p\in\mathcal O$. 

On $\mathcal U$, $X\neq0$ and the vector fields $V_I,IV_I,JV_J=- JV_I$, and $IJV_I$ are linearly independent. Over $\mathcal{U}$, define the following rank $4$ subbundle
$$
\mathcal E:=\operatorname{span}\{V_I,IV_I,JV_I,IJV_I\}.
$$

We have shown above that for every $p\in\mathcal O\cap\mathcal U$, $V_I(p)\in\operatorname{Im}\sigma_p$ (Equation \eqref{eqn:im_sigma}). Since $\operatorname{Im}\sigma_p$ is $I_p$- and $J_p$-invariant and has dimension $4$ (cf. Subsection \ref{subs:GK}), it follows that
\begin{equation} \label{eqn:im_sigma2}
 \operatorname{Im}\sigma_p=\mathcal E_p,
\qquad p\in\mathcal O\cap\mathcal U.   
\end{equation}
\begin{lem} \label{lem:E_J_invariant}
For either $c=\pm1$, the rank-four distribution $\mathcal E$ on $\mathcal U$ is $J$-invariant.
\end{lem}

\begin{proof}
If $c=1$, then $[I,J]=0$ by Proposition \ref{prop:vanishin_poisson}. The $J$-invariance of $\mathcal E$ follows immediately from $IJ=JI$.
Suppose now that $c=-1$. We first show that $\mathcal O$ is dense in $M$. Otherwise, there would exist a nonempty open subset $\mathcal V\subset M\setminus\overline{\mathcal O}$ on which $\sigma$ vanishes identically, then since its $(2,0)$ component with respect to $I$, for instance, is holomorphic, $\sigma$ would vanish everywhere on $M$. Thus $I$ and $J$ commute, and $IJ$ defines a smooth involution of $TM$ such that
\[
TM= TM_+ \oplus TM_-
\]
where $TM_\pm= \ker(IJ \mp \mathrm{Id}) = \ker(I \pm J)=\mathrm{Im}(I \mp J)$ are smooth subbundles \cite[Theorem 4]{AG}. Therefore $\dim_{\mathbb R}\operatorname{Im}(\pi_\pm)$ is constant on $M$, and with the same argument as before, the generalized Darboux theorem applies. Therefore, locally, the generalized K\"ahler structure is defined by $d_H$-closed pure spinors. Hence \cite[Proposition 4.3]{AFSU} applies locally and yields
$
4V_I=\theta_I^\sharp-\theta_J^\sharp\in \mathrm{Im}\sigma=0,
$
contradicting $|V_I|=1$. Therefore, $\mathcal O$ is dense in $M$ and hence $\mathcal{O} \cap \mathcal{U}$ is dense in $\mathcal{U}$. Since $\mathcal{E}$ is $J$-invariant on the dense subset $\mathcal{O} \cap \mathcal{U}$, it is $J$-invariant on $\mathcal{U}$ by continuity.
\end{proof}
\subsection{The normal form of $J$}
In this subsection, we show that every point of $\mathcal U$ admits an open neighborhood contained in $\mathcal U$ and an adapted orthonormal frame in which $J$ takes a canonical form. We then extend this expression pointwise to $M\setminus\mathcal U$.

Recall that on $\mathcal{U}$ the vector field $X=JV_J-b IV_I=c JV_I -b IV_I$ is nonzero. Consequently, its norm satisfies $|X|=\sqrt{1-b^2}>0$ (cf. Equation \eqref{eqn:norm_X}). Set $s:=|X|$ and
\begin{equation} \label{eqn:basis}
e_0:=V_I, \ \ e_1:=IV_I, \ \ e_2:=\frac{cX}{s}, \ \ e_3:=Ie_2,
\end{equation}
where $c\in \{\pm1 \}$ is determined by $V_J=c V_I$ (Proposition \ref{prop:linear_dependence}). The vector fields $\{e_0,\dots,e_3\}$ are orthonormal, and
$$
V_I= e_0, \ \ IV_I= e_1, \, \, 
J V_I= Je_0=  c JV_J= c (b IV_I +X)=cb e_1 + s e_2, \ \, IJV_I=IJe_0=-cb\,e_0+s\,e_3.
$$
Thus, $\mathcal{E}=\mathrm{span}\{e_0,\dots,e_3\}$, which is both $I$- and $J$-invariant by Lemma \ref{lem:E_J_invariant}. We can locally choose a unit vector field $e_4\in\mathcal E^\perp$ and set $e_5:=Ie_4$. Thus, $\{e_0,\ldots,e_5\}$ is a local orthonormal frame on $\mathcal U$. Define $u:=-s\,e_1+cb\,e_2.$ Then $u$ is a unit vector field orthogonal to both $e_0$ and $Je_0$. Therefore, we may write
$\mathcal E=\operatorname{span}\{e_0,Je_0=cb\, e_1 + s \,e_2,u=-s\,e_1+cb\,e_2,e_3\}.
$
The $J$-invariance of $\mathcal E$ (Lemma \ref{lem:E_J_invariant}), together with the fact that $g$ is $J$-Hermitian, implies that
$Ju=\pm e_3.$
\begin{lem}
On $\mathcal{U}$, $Je_3=c u$.
\end{lem}
\begin{proof}
If $c=1$, Proposition \ref{prop:vanishin_poisson} gives $[I,J]=0$. Since $[I,J]e_0=0$, we have $e_3=\tfrac{1}{s} Je_1+\tfrac{b}{s}\, e_0$. Applying $J$ once more gives $Je_3=- s\, e_1+b\, e_2=u$, and hence $Je_3=u$.

Suppose instead that $c=-1$. At every point of the dense subset $\mathcal O\cap\mathcal U$, we must have $Ju=e_3$. Indeed, assume by contradiction that $Ju=-e_3$. Since $c=-1$, $Je_0=-be_1+se_2$ and $u=-s e_1-b e_2$, which imply that $e_1=-b Je_0-su$. Therefore,
$$
JIe_0=Je_1=b e_0+s e_3=IJe_0,
$$
and hence $[I,J]e_0=0$. This is impossible because, on $\mathcal{O}\cap \mathcal{U}$, $e_0=V_I \in \Im([I,J])$ (Equation \eqref{eqn:im_sigma}) and $\Im([I,J])\cap \ker ([I,J])=0$ (Equation \ref{eqn:splitting2}). Since $\mathcal{U}\cap\mathcal{O}$ is dense in $\mathcal U$, continuity yields $Ju=e_3$ throughout $\mathcal U$.
\end{proof}
Since $e_1=IV_I=cb\,Je_0-s\,u,$ and $e_2=s\,Je_0+cb\,u,$
we obtain
\begin{equation} \label{eqn:J_normal_form}
\begin{split}
Je_0&=cb\,e_1+s\,e_2, \\
Je_1&=-cb\,e_0+cs\,e_3,\\
Je_2&=-s\,e_0-b\,e_3, \\
Je_3&=-cs\,e_1+b\,e_2.
\end{split} 
\end{equation}
\begin{thm} \label{thm:normal_form_of_J}
Let $(M^6,I,J,g)$ be a connected compact GKBHE manifold with $r=1$ and $V_I=c V_J$. Every point $p\in\mathcal U$ admits an open neighborhood contained in $\mathcal U$ carrying an $I$-unitary orthonormal frame $\{e_0,\dots,e_5\}$, where $\{e_0,\dots,e_3\}$ is the $I$-unitary orthonormal frame of $\mathcal E$ defined in \eqref{eqn:basis}, while $\{e_4,e_5=Ie_4\}$ is an orthonormal frame of $\mathcal E^\perp$, such that
\begin{equation} \label{eqn:matrix_J}
J=\begin{pmatrix}
0 & -cb & -s & 0 & 0 & 0\\
cb & 0 & 0 & -cs & 0 & 0\\
s & 0 & 0 & b & 0 & 0\\
0 & cs & -b & 0 & 0 & 0\\
0 & 0 & 0 & 0 & 0 & c\\
0 & 0 & 0 & 0 & -c & 0
\end{pmatrix},
\end{equation}
where $b= g(JV_J,IV_I)$ and $s=\sqrt{1-b^2}$. Furthermore, for every $p\in M\setminus\mathcal U$, there exists an $I$-unitary orthonormal basis of $T_pM$, with $e_0=V_I$, in which $J_p$ is represented by the same matrix \eqref{eqn:matrix_J}, with $s(p)=0$ and $b(p)=g_p(JV_J(p),IV_I(p))\in \{\pm1\}$.
\end{thm}

\begin{proof}
The expression of $J$ on $\mathcal E$ follows from \eqref{eqn:J_normal_form}. Since $\mathcal E$ is invariant under both $I$ and $J$, so is $\mathcal E^\perp$. Hence $Je_4=\pm e_5$. The orientations induced by $I$ and $J$ on $\mathcal E$ differ by the sign $-c$. Indeed, the $J$-oriented frame $\{e_0,Je_0,u,Ju\}$ is related to the $I$-oriented frame ${e_0,e_1,e_2,e_3}$ by
$$
Je_0=cb\,e_1+s\,e_2,\qquad u=-s\,e_1+cb\,e_2,\qquad Ju=-c\,e_3,
$$
and the corresponding change-of-basis determinant is $-c$. Since $I$ and $J$ induce the same orientation on $M$, 
$$
Je_4=-c\,e_5,\qquad Je_5=c\,e_4.
$$
Together with \eqref{eqn:J_normal_form}, this gives \eqref{eqn:matrix_J}.

We now prove the statement at a point $p\in M\setminus\mathcal U$. Since $\mathcal U$ is dense, there exists a sequence $(p_k)_k\subset\mathcal U$ converging to $p$. For each $k$, let $B_k$ be a $I$-unitary orthonormal basis of $T_{p_k}M$ in which $J_{p_k}$ has the form \eqref{eqn:matrix_J}. Since $p_k\to p$, all but finitely many $p_k$ lie in a trivializing neighborhood of the unitary frame bundle of $(TM,I,g)$. Since its fiber $U(3)$ is compact, after passing to a subsequence, the frames $(p_k,B_k)$ converge in the unitary frame bundle to a frame $(p,B)$, where $B=\{e_0,\ldots,e_5\}$ is an $I$-unitary orthonormal basis of $T_pM$. Since the first vector of $B_k$ is $V_I(p_k)$, the smoothness of $V_I$ implies that $e_0=V_I(p)$. Writing $B_k=\{e_0^k,\ldots,e_5^k\}$, the convergence of the frames and the smoothness of $(J,g)$ imply that $g_{p_k}(J_{p_k}e_b^k,e_a^k)\to g_p(J_pe_b,e_a)$ for every $a,b$. Thus, the matrices of $J_{p_k}$ with respect to $B_k$ converge entrywise to the matrix of $J_p$ with respect to $B$. Moreover, since $b$ and $X$ are smooth,
$$
b(p_k)\longrightarrow b(p),\qquad s(p_k)=|X(p_k)|\longrightarrow |X(p)|=0.
$$
By \eqref{eqn:norm_X}, $|X(p)|^2=1-b(p)^2$, and hence $b(p)=\pm 1$. Passing to the limit in \eqref{eqn:matrix_J}, we conclude that the matrix of $J_p$ with respect to $B$ is given by the same expression, with $s=0$ and $b=\pm1$.
\end{proof}
Proposition \ref{prop:vanishin_poisson} shows that $\sigma$ vanishes identically when $c=1$. Suppose now that $c=-1$. We have already shown that $V_I\in\operatorname{Im}(\sigma)$ on $\mathcal O$ (Equation \eqref{eqn:im_sigma}) and that $\operatorname{Im}(\sigma)=\mathcal E$ on $\mathcal O\cap\mathcal U$ (Equation \eqref{eqn:im_sigma2}). The following corollary identifies $\mathcal O$ with $\mathcal U$; in particular, the latter description of $\operatorname{Im}(\sigma)$ holds throughout $\mathcal U$.
\begin{cor}
Let $(M^6,I,J,g)$ be a connected compact GKBHE manifold with $r=1$. If $c=-1$, then $\mathcal O=\mathcal U$ and $\operatorname{Im}(\sigma)=\mathcal E$ on $\mathcal U$.
\end{cor}
\begin{proof}
By Theorem \ref{thm:normal_form_of_J}, when $c=-1$, $|[I,J]|^2=16s^2=16(1-b^2)=16|X|^2.$ Thus $[I,J]_p\neq0$ if and only if $X(p)\neq0$. By the definitions of $\mathcal O$ and $\mathcal U$, this proves that $\mathcal O=\mathcal U$. The description of $\operatorname{Im}(\sigma)$ then follows from Equation \eqref{eqn:im_sigma2}.
\end{proof}

\subsection{Transverse geometry}
In this subsection, we introduce a smooth basic function $F$ and derive several identities that it satisfies. In the next section, we shall show that these identities are sufficient to construct a GKBHE structure in the converse direction.

Let $(M^6,I,J,g)$ be a connected compact GKBHE manifold with $r=1$. Define
\begin{equation}\label{eqn:F}
F:=e^{-f}(1+cb),
\end{equation}
where $b=g(IV_I,JV_J)$, $f$ is the normalized soliton potential and $c=\pm1$ is defined by the relation $V_J=c V_I$ (see Proposition \ref{prop:linear_dependence}). Since $e^{-f}>0$, $b(p)\in[-1,1]$, and $c=\pm 1$, the function $F$ is nonnegative and it is strictly positive on $\mathcal U$. Since $f$ is basic by item \ref{item:basic} in Lemma \ref{lem:a_priori_symmetries} and $b$ is basic by item \ref{item:GKBHE-coefficients-basic} in Proposition \ref{prop:GKBHE-horizontal-symmetries}, the function $F$ is basic with respect to the foliation $\mathcal F_I$.
\begin{lem} \label{claim:db}
The function $b$ satisfies $db=(b+c)df+2(IX)^\flat$.
\end{lem}
\begin{proof}
We first prove the identity on $\mathcal U$. Fix $p\in\mathcal U$ and choose an open neighborhood $\mathcal W\subset\mathcal U$ of $p$ on which the adapted frame $\{e_0,\dots,e_5\}$ of Theorem \ref{thm:normal_form_of_J} is defined. We denote its dual coframe by $\{e^0,\dots,e^5\}$. Let $H_J$ denote the Bismut torsion of $(J,g)$. Since $H_J=-H$, the trace identity for the Lee form (Equation \eqref{eqn:Lee_form}) gives, for every vector field $Y$,
\begin{equation}\label{eqn:lee_form1}
\theta_J(Y)
=-\frac12\sum_{i=0}^5H_J(JY,e_i,Je_i)
=\frac12\sum_{i=0}^5H(e_i,Je_i,JY).
\end{equation}
Furthermore, by Equation \eqref{eqn:torsionH}, $H$ is given by 
$$
H=H^T+\eta_I\wedge d\eta_I
  =I df\wedge\omega_I^T+\eta_I\wedge d\eta_I.
$$
In the adapted frame, $\eta_I=e^1$ and $\omega_I^T=e^{23}+e^{45}.$ Since $f$ is basic and $X(f)=0$ by item \ref{item:GKBHE-horizontal-f-invariant} in Proposition \ref{prop:GKBHE-horizontal-symmetries}, we have $df(e_0)=df(e_1)=df(e_2)=0.$ Consequently, $Idf=df(e_3)e^2+df(e_5)e^4-df(e_4)e^5,$ and hence
$$
H=e^1\wedge d\eta_I
  +df(e_3)e^{245}
  -df(e_4)e^{235}
  +df(e_5)e^{234}.
$$
Substituting the normal form of $J$ into \eqref{eqn:lee_form1}, and using that $H$ has no component involving $e_0$, we obtain
$$
\theta_J(Y)
=cs\,H(e_1,e_3,JY)
-b\,H(e_2,e_3,JY)
-c\,H(e_4,e_5,JY).
$$
Evaluating the preceding expression for $\theta_J$ on the horizontal frame vectors therefore gives
\begin{equation}\label{eqn:Theta_J}
\begin{split}
\theta_J(e_2)&=0,\\
\theta_J(e_3)&=s\,d\eta_I(e_4,e_5)-cb\,df(e_3),\\
\theta_J(e_4)&=-s\,d\eta_I(e_3,e_5)-cb\,df(e_4),\\
\theta_J(e_5)&=s\,d\eta_I(e_3,e_4)-cb\,df(e_5).\\
\end{split}
\end{equation}
Since $d\eta_I$ is $I$-invariant, $d\eta_I(e_3,e_5)=d\eta_I(e_2,e_4),
$ and $d\eta_I(e_3,e_4)=-d\eta_I(e_2,e_5).$ Furthermore, Equation \eqref{eq:BHE-curvature-traces} gives $d\eta_I(e_2,e_3)+d\eta_I(e_4,e_5)=-2.$
Combining this identity with \eqref{eqn:Theta_J}, we obtain
\begin{equation}\label{eqn:theta_J2}
\begin{split}
\left.\theta_J\right|_{\mathcal H}
=&\bigl(-2s-s\,d\eta_I(e_2,e_3)-cb\,df(e_3)\bigr)e^3+\bigl(-s\,d\eta_I(e_2,e_4)-cb\,df(e_4)\bigr)e^4+\bigl(-s\,d\eta_I(e_2,e_5)-cb\,df(e_5)\bigr)e^5.
\end{split}
\end{equation}
On the other hand, item \ref{item:GKBHE-curvature-potentials} in Proposition \ref{prop:GKBHE-horizontal-symmetries} gives $db=-\iota_Xd\eta_I.$ Since $X=cs \, e_2$ and $IX=cs\,e_3$, it follows that $
c\,db=-s\,\iota_{e_2}d\eta_I,$ and $-2c(IX)^\flat=-2s\,e^3.$
Thus \eqref{eqn:theta_J2} can be written as
\begin{equation}\label{eqn:horizontal_lee_J}
\left.\theta_J\right|_{\mathcal H}
=c\,db-2c(IX)^\flat-cb\,df.
\end{equation}
Finally, by the definition of the canonical vector field and Proposition \ref{prop:linear_dependence}, $\theta_J=df+2V_J^\flat=df+2c \eta.$ Since $\mathcal H$ is orthogonal to $V_I$ and $f$ is basic, $
\left.\theta_J\right|_{\mathcal H}=df.$ Comparing this identity with \eqref{eqn:horizontal_lee_J}, we obtain
$
df=c\,db-2c(IX)^\flat-cb\,df.
$
Since $c^2=1$, this is equivalent to the claimed identity.
The $1$-form $db-(b+c)df-2(IX)^\flat$ is smooth on $M$ and vanishes on the dense subset $\mathcal U$. It therefore vanishes identically on $M$.
\end{proof}
\begin{lem} \label{claim:gradient}
It holds that
$
X=-\frac{c}{2}I^T\grad_{g_K^T}F.
$
\end{lem}
\begin{proof}
Since $F$ is basic, its transverse gradient is well defined. Differentiating \eqref{eqn:F} and using Lemma \ref{claim:db}, we obtain
\begin{equation} \label{eqn:dF}
\begin{split}
dF
&=d\bigl(e^{-f}(1+cb)\bigr)\\
&=-e^{-f}(1+cb)df+ce^{-f}db\\
&=-e^{-f}(1+cb)df
+ce^{-f}\bigl((b+c)df+2(IX)^\flat\bigr)\\
&=2ce^{-f}(IX)^\flat.
\end{split}
\end{equation}
Since $\iota_X\omega_I^T=(IX)^\flat$ and $\omega_K=e^{-f}\omega_I^T$, this identity can be written as
$dF=2c\,\iota_X\omega_K^T
   =2c\,g_K(I^TX,\cdot).$
Therefore
$
\grad_{g_K}F=2cI^TX.
$
Applying $I^T$ and using $(I^T)^2=-\operatorname{Id}$ and $c^2=1$, we conclude that $X=-\frac{c}{2}I^T\grad_{g_K}F.$
\end{proof}
\begin{lem} \label{claim:laplacian}
It holds that $\Delta_{g_K}F=-8cb.$ Moreover, on the open set $\{F>0\}$, $\Delta_{g_K}\log F=-\frac{R_K}{2}.$
\end{lem}
\begin{proof}
Since the Bismut torsion of $(J,g)$ is $-H$ and $JV_I=cJV_J$ is $\nabla^{B,J}$-parallel by item \ref{item:JV_J} in Lemma \ref{lem:a_priori_symmetries}, it follows that $\nabla_Y(JV_I)=\tfrac12\, (\iota_YH(JV_I))^\sharp.$ On the other hand, since $\nabla^{B,I}=\nabla+\frac12H$ preserves $I$, we have $(\nabla_YI) Z=-\frac12\big(H(Y,IZ)^\sharp- I H(Y,Z)^\sharp\big).$
It follows that
\begin{align*}
\nabla_Y(IJV_I)=(\nabla_YI)(JV_I)+I\nabla_Y(JV_I)=I\left(\iota_YH(JV_I)\right)^\sharp-\tfrac12 \left(\iota_YH(IJV_I)\right)^\sharp.
\end{align*}
Taking the trace with respect to a local orthonormal frame ${e_i}$ gives
\begin{align*}
\operatorname{div}(IJV_I)
=\sum_i g\bigl(I\left(\iota_{e_{i}}H(JV_I)\right)^\sharp,e_i\bigr)
=\sum_i H(e_i,Ie_i,JV_I)
=2\theta_I(IJV_I),
\end{align*}
where the last equality follows from Equation \eqref{eqn:Lee_form}.
Since $JV_I=cJV_J=cb\,IV_I+cX$, we have $IJV_I=-cb\,V_I+cIX.$ Using $\theta_I=df+2 \eta$ and $V_I(f)=0$ (item \ref{item:basic} in Lemma \ref{lem:a_priori_symmetries}), we obtain
\begin{equation}\label{eqn:div_IJVI}
\operatorname{div}(IJV_I)
=2c\,df(IX)-4cb.
\end{equation}
On the other hand, since $V_I(b)=0$ (item \ref{item:GKBHE-coefficients-basic} in Proposition \ref{prop:GKBHE-horizontal-symmetries}) and $V_I$ is Killing (item \ref{item:V_I} in Lemma \ref{lem:a_priori_symmetries}),
$$
\operatorname{div}(IJV_I)
=\operatorname{div}(-cbV_I+cIX)
=c\,\operatorname{div}(IX).
$$
Thus \eqref{eqn:div_IJVI} yields
\begin{equation}\label{eqn:div_IX}
\operatorname{div}(IX)=2df(IX)-4b.
\end{equation}
By Equation \eqref{eqn:dF} in the proof of Lemma \ref{claim:gradient}, $\grad_gF=2ce^{-f}IX.$ Hence, using \eqref{eqn:div_IX}, we compute
\begin{align*}
\Delta F=\operatorname{div}\bigl(2ce^{-f}IX\bigr)=2ce^{-f}\bigl(\operatorname{div}(IX)-df(IX)\bigr)=2ce^{-f}df(IX)-8ce^{-f}b.
\end{align*}
Using Equation \eqref{eqn:dF} once again, we have $2ce^{-f}df(IX)=g^T(dF,df)$, and therefore
\begin{equation}\label{eqn:weighted_laplacian_F}
\Delta F-g^T( df,dF)=-8ce^{-f}b.
\end{equation}
Since $F$ is basic, $\Delta F= \Delta_{g_I^T} F$, and the conformal transformation law for the Laplacian in dimension $4$ gives
\(\Delta_{g_K}F=e^f(\Delta_{g^T}F-g^T( df,dF))=-8cb.\)
Finally, the identity $\grad_{g_K}F=2cI^TX$ proved in Lemma \ref{claim:gradient}, together with $|X|_g^2=1-b^2$ (cf. Equation \eqref{eqn:norm_X}), implies
$$
|dF|_{g_K}^2
=4|X|_{g_K}^2
=4e^{-f}(1-b^2)
=4F(1-cb).
$$
Therefore, wherever $F>0$,
\begin{align*}
\Delta_{g_K}\log F=\frac{\Delta_{g_K}F}{F}
-\frac{|dF|_{g_K}^2}{F^2}=-\frac{8cb}{F}-\frac{4(1-cb)}{F}=-\frac{4(1+cb)}{F}=-4e^f=-\frac{R_K}{2},
\end{align*}
where, in the last equality, we used $F=e^{-f}(1+cb)$ and $R_K=8e^f$ (Theorem \ref{thm:BHE-six-to-four}).
\end{proof}
\begin{thm}\label{thm:transverse_geometry}
Let $(M^6,I,J,g)$ be a connected compact GKBHE manifold with $r=1$. Then $F=e^{-f}(1+cb)$ is a smooth nonnegative basic function satisfying $X=-\frac{c}{2}I^T\grad_{g_K}F.$ In particular, $I^T\grad_{g_K}F$ is $I^T$-holomorphic. Moreover, the subset $\{F>0\}$ is open and dense, and $\rho_K=-dd^c\log F$ on $\{F>0\}$.
\end{thm}
\begin{proof}
The first part of the theorem follows from the definition of $F$ and Lemma \ref{claim:gradient}. The fact that $I^T\grad_{g_K}F$ is $I^T$ holomorphic is an obvious consequence of item \ref{item:GKBHE-transverse-automorphisms} in Proposition \ref{prop:GKBHE-horizontal-symmetries}. Now we prove that $\{F>0\}$ is a dense open subset of $M$, and
$
\rho_K=-dd^c\log F
$
on $\{F>0\}$.
Since $\mathcal U\subset\{F>0\}$ and $\mathcal U$ is dense in $M$ by Lemma \ref{lem:density_U}, the open set $\{F>0\}$ is also dense. 

To show that $\rho_K=-dd^c\log F$ we work on $\mathcal U$. On $\{F>0\}$, consider the basic real $(1,1)$-form $\beta:=\rho_K+dd^c\log F.$ By Lemma \ref{claim:laplacian}, $$
\mathrm{tr}_{\omega_K}\beta
=R_K+2\Delta_{g_K}\log F=0.
$$
Thus $\beta$ is anti-self-dual with respect to the transverse K\"ahler metric. By items \ref{item:GKBHE-horizontal-projectable} and \ref{item:GKBHE-horizontal-f-invariant} in Proposition \ref{prop:GKBHE-horizontal-symmetries}, $X(F)=0$. Since $X$ is  $I^T$-holomorphic, this implies $\mathcal L_Xd^c\log F=0.$ Therefore, by Cartan's formula, $
\iota_Xdd^c\log F=-d\bigl(\iota_Xd^c\log F\bigr).$ Using $d^c\log F=-\frac{I^TdF}{F}$, Equation \eqref{eqn:dF}, and $|X|^2=1-b^2$, we obtain
$$
\iota_Xd^c\log F
=-\frac{dF(I^TX)}{F}
=-\frac{2ce^{-f}|X|^2}{F}
=-\frac{2ce^{-f}(1-b^2)}{e^{-f}(1+cb)}
=2(b-c).
$$
Since $c$ is constant, it follows that $$
\iota_Xdd^c\log F=-2\,db.
$$
On the other hand, Equation \eqref{eqn:diff_curvatures}, together with Theorem \ref{thm:BHE-six-to-four} gives
$$
\iota_X\rho_K=-2\iota_Xd\eta_I=2\,db.
$$
Consequently, $\iota_X\beta=0$ on $\mathcal U$. Since $X$ is nowhere vanishing on $\mathcal U$ and a nonzero anti-self-dual $2$-form in dimension four is nondegenerate, we conclude that $\beta=0$ on $\mathcal U$. Finally, $\beta$ is smooth on $\{F>0\}$, and $\mathcal U$ is dense in this set. Hence $\beta=0$ throughout $\{F>0\}$.
\end{proof}
\begin{rmk}
Note that the locus $F^{-1}(0)$, and hence $M \setminus \mathcal{U}$, is nonempty. This can be proved as follows: Let $p$ be a point where $F$ attains its minimum and assume by contradiction that $F(p)>0$. \par
Since $p$ is a minimum $dF(p)=0$ and hence $X(p)=0$. This implies that $(1-b(p)^2)=0$ and hence $cb(p)=\pm1$. If $cb(p)=-1$, then the definition of $F$ (see \eqref{eqn:F}), would imply that $F(p)=0$. Therefore $cb(p)=1$. Equation \eqref{eqn:weighted_laplacian_F}, then reads as 
\[
\Delta F(p)=-8 e^{-f(p)}<0.
\]
But since $p$ is a minimum, $\Delta F(p)\ge 0$, giving a contradiction.
\end{rmk}
\begin{ex}
Let \(0<q<1\), let
\[
H_q=
\frac{\mathbb C^2\setminus\{0\}}
{\langle(z^1,z^2)\mapsto(qz^1,qz^2)\rangle},
\]
and let \((E,I_E,g_E)\) be a flat elliptic curve. On
\(M=H_q\times E\cong\mathrm{SU}(2)\times\mathbb T^3\), consider the metric
\[
g=
\frac{dz^1\,d\bar z^1+dz^2\,d\bar z^2}
{|z^1|^2+|z^2|^2}+g_E.
\]
We write a real tangent vector as \((v,v_E)=(v^1,v^2,v_E)\), where
\(v^j=dz^j(v)\) and \(v_E\in TE\). Both GKBHE structures have the same
first complex structure
\[
I(v^1,v^2,v_E)=(iv^1,iv^2,I_Ev_E).
\]
For \(c=1\), define
\[
J_{1}(v^1,v^2,v_E)=(iv^1,-iv^2,-I_Ev_E).
\]
For \(c=-1\), set
\[
a=\frac{|z^1|^2-|z^2|^2}{|z^1|^2+|z^2|^2},
\qquad
h=\frac{2z^1z^2}{|z^1|^2+|z^2|^2},
\]
and define
\[
J_{-1}(v^1,v^2,v_E)
=
\bigl(
ia\,v^1+ih\,\overline{v^2},
ia\,v^2-ih\,\overline{v^1},
I_Ev_E
\bigr).
\]
All these expressions descend to \(M\).

Under the Lie group identification
\(M\cong\mathrm{SU}(2)\times\mathbb T^3\), the metric is bi-invariant,
\(I\) is left-invariant, and \(J_{1}\) and \(J_{-1}\) are right-invariant.
Their Bismut connections are flat, and the torsions of \(J_{1}\) and
\(J_{-1}\) are both opposite to that of \(I\). Thus
\((g,I,J_{1})\) and \((g,I,J_{-1})\) are GKBHE structures. Moreover, all three
complex structures induce the same orientation.

The normalized soliton potential is constant, \(f=f_0\).
The radial vector field
\[
T=\sum_{j=1}^{2}
\left(
z^j\partial_{z^j}+\bar z^j\partial_{\bar z^j}
\right)
\]
descends to a unit vector field on \(M\). The Lee forms satisfy
\[
\theta_I=\theta_{J_{1}}=-2T^\flat,
\qquad
\theta_{J_{-1}}=2T^\flat.
\]
Consequently,
\[
V_I=-T,\qquad V_{J_{1}}=-T,\qquad V_{J_{-1}}=T,
\]
which identifies the two choices of \(c\).

Their commutator Poisson tensors
\(\sigma_{*}=\frac12[I,J_{*}]g^{-1}\), for $*=1,-1$ behave differently.
The structures \(I\) and \(J_1\) commute, so $\sigma_1=0.$
For \(J_{-1}\), direct computation gives
\[
\begin{aligned}
{}[I,J_{-1}]\partial_{\bar z^1}
&=\frac{4z^1z^2}{|z^1|^2+|z^2|^2}\,\partial_{z^2},\\
[I,J_{-1}]\partial_{\bar z^2}
&=-\frac{4z^1z^2}{|z^1|^2+|z^2|^2}\,\partial_{z^1}.
\end{aligned}
\]
Since
\(g^{-1}(dz^j)=2(|z^1|^2+|z^2|^2)\partial_{\bar z^j}\), the
\(I\)-holomorphic Poisson tensor is
\[
\sigma_{-1}^{2,0}
=4z^1z^2\,\partial_{z^1}\wedge\partial_{z^2},
\qquad
\sigma_{-1}=\sigma_{-1}^{2,0}+\overline{\sigma_{-1}^{2,0}}.
\]
This bivector is invariant under the contraction and hence is globally
defined on \(M\). It has real rank four on \(\{z^1z^2\neq0\}\), where
its image is tangent to the Hopf factor, and vanishes on
\(\{z^1z^2=0\}\).

Finally, writing \(b_*=g(IV_I,J_{*}V_{J_*})\), for $*=1,-1$, we obtain
\[
b_{1}=a,\qquad b_{-1}=-a.
\]
The two structures therefore determine the same function
\[
F_{1}=e^{-f_0}(1+b_{1})
=
F_{-1}=e^{-f_0}(1-b_{-1})
=
2e^{-f_0}\frac{|z^1|^2}{|z^1|^2+|z^2|^2},
\]
we denote this function by $F$. 
For both structures, the pointwise horizontal rank is one outside the
two disjoint tori
\[
D_1=\{z^1=0\}\cong\mathbb T^4,
\qquad
D_2=\{z^2=0\}\cong\mathbb T^4,
\]
and zero on \(D_1\sqcup D_2\). More precisely, 
\[
\{F=0\}=D_1,
\qquad
\{F=2e^{-f_0}\}=D_2.
\]
For \(c=-1\), \(D_1\sqcup D_2\) is also exactly the zero locus of the
commutator Poisson tensor.
\end{ex}
\section{The converse construction} \label{section:converse}
In this section, we establish local and global converses to
Theorem~\ref{thm:transverse_geometry}. We start with a K\"ahler surface
$(S,I^T,g_K,\omega_K)$ endowed with a smooth nonnegative function $F$
such that $I^T\grad_{g_K}F$ is real holomorphic and
\[
\rho_K=-dd^c\log F
\qquad\text{on }S_+:=\{F>0\}.
\]
The crucial point is that these assumptions imply the box
equation~\eqref{eq:BHE-four-dimensional-equation} with $\alpha=0$.

\subsection{Preliminary identities}
In this subsection, we collect several differential identities
which will be useful to prove the converse construction. 

\begin{prop}\label{prop:kahler-identities}
Let $(S,I^T,g_K,\omega_K)$ be a connected K\"ahler orbifold of dimension $2n$, $n \ge 2$, and let $F$ be a smooth nonnegative function such that $S_+:=\{F>0\}\neq\varnothing.$
Let $\rho$ and $R$ denote the Ricci form and scalar curvature of $g_K$.
Assume that $I^T\grad F$ is holomorphic on $S$ and that
\[
\rho=-dd^c\log F
\]
on $S_+$. Then $S_+$ is dense in $S$, and the following identities hold:
\begin{enumerate}[label=\textup{(\alph*)}, ref=(\alph*)]
    \item\label{item:kahler-rho-hamiltonian}
    $\displaystyle \iota_{(I^T\grad F)}\rho=\frac12\,d(\Delta F)$ on $S$;
    
    \item\label{item:kahler-delta-energy}
    $\displaystyle d(\Delta F)=2d\left(\frac{|dF|^2}{F}\right)$ on $S_+$;
    
    \item\label{item:kahler-box}
    $\displaystyle d(RF)=-d(\Delta F)$ on $S$, hence $RF+\Delta F$ is constant on $S$;

    \item\label{item:potential}
    $ d^c\log F\wedge\rho \wedge \omega_K^{[n-2]}=-\frac12\,d^cR\wedge\omega_K^{[n-1]}$ on $S_+$.
    
    \item\label{item:kahler-ricci-square}
    $\displaystyle \rho\wedge\rho\wedge \omega_K^{[n-2]}=\frac12\,dd^cR\wedge\omega^{[n-1]}_K$ on $S$, where $\omega_K^{[r]}=\frac{\omega_K^r}{r!}$. For $n=2$, $\displaystyle \rho\wedge\rho=\frac12\,dd^cR\wedge\omega_K$ on $S$.
\end{enumerate}
Furthermore, if $S$ is compact and $R>0$, then $S_+$ is connected.
\end{prop}

\begin{proof}
We first show that $S_+$ is dense. Otherwise, $F$ would vanish on a nonempty open subset of $S$, and hence so would $I^T\grad F$. Since $I^T\grad F$ is holomorphic and $S$ is connected, it would follow that $I^T\grad F$ vanishes identically. Thus $F$ would be constant and, since it vanishes somewhere, identically zero, contradicting $S_+\neq\varnothing$.

The remaining computations are local. We work in a uniformizing chart
and suppress the pullback notation. We now prove \ref{item:kahler-rho-hamiltonian}. Since
\[
\iota_{(I^T\grad F)}\omega_K=-dF,
\]
we have $\mathcal L_{(I^T\grad F)}\omega_K=0$. Moreover, $I^T\grad F$ is holomorphic by hypothesis, so $\mathcal L_{(I^T\grad F)}I^T=0$. Hence $I^T\grad F$ preserves both $I^T$ and $\omega_K$, and is therefore Killing. With our conventions $\Delta=-d^*d$ and $\nabla^*\nabla=-\operatorname{tr}\nabla^2$, the standard Bochner identities for a function and a Killing vector field give
\[
\nabla^*\nabla(\grad F)
=-\grad(\Delta F)-\operatorname{Rc}(\grad F),
\qquad
\nabla^*\nabla I^T\grad F=\operatorname{Rc}(I^T\grad F);
\]
see, for instance, \cite[Equations (14.2) and (16.4)]{Biq}. Since $\nabla I^T=0$, the rough Laplacian commutes with $I$. Moreover, the Ricci tensor of a K\"ahler metric is $I$-invariant, and hence $\operatorname{Rc}I^T=I^T\operatorname{Rc}$.
Applying $I^T$ to the first Bochner identity and comparing it with the second one, we obtain
\[
-I^T\grad(\Delta F)-I^T\operatorname{Rc}(\grad F)
=I^T\operatorname{Rc}(\grad F).
\]
Consequently,
\[
\grad(\Delta F)=-2\operatorname{Rc}(\grad F).
\]
Using $\rho(Y,Z)=\operatorname{Rc}(I^TY,Z)$, for every vector field $Y$ we therefore have
\[
(\iota_{(I^T\grad F)}\rho)(Y)
=\rho( I^T \grad F,Y)
=-\operatorname{Rc}(\grad F,Y)
=\frac12\,d(\Delta F)(Y),
\]
which proves \ref{item:kahler-rho-hamiltonian}.
We next prove \ref{item:kahler-delta-energy}. We work on $S_+$. Since $(I^T\grad F)F=0$ and $I^T\grad F$ is holomorphic,
\[
\mathcal L_{(I^T\grad F)}d^c\log F=d^c\mathcal L_{(I^T\grad F)}\log F=0.
\]
Cartan's formula and $\rho=-dd^c\log F$ then give
\[
\iota_{(I^T\grad F)}\rho
=-\iota_{(I^T\grad F)}dd^c\log F
=d\bigl(\iota_{(I^T\grad F)}d^c\log F\bigr).
\]
We have
\[
\iota_{(I^T\grad F)}d^c\log F
=-\frac{dF(I^T(I^T\grad F))}{F}
=\frac{|dF|^2}{F}.
\]
Thus
\begin{equation} \label{eqn:contraction}
    \iota_{(I^T\grad F)}\rho=d\left(\frac{|dF|^2}{F}\right),
\end{equation}
and \ref{item:kahler-delta-energy} follows from \ref{item:kahler-rho-hamiltonian}. To prove \ref{item:kahler-box}, we work in $S_+$. Trace the identity $\rho=-dd^c\log F$ with respect to $\omega_K$. We obtain
\[
R=-2\Delta\log F
=-2\frac{\Delta F}{F}
+2\frac{|dF|^2}{F^2}.
\]
Multiplying by $F$ and differentiating yields
\[
d(RF)
=-2d(\Delta F)
+2d\left(\frac{|dF|^2}{F}\right).
\]
By \ref{item:kahler-delta-energy}, the last term is $d(\Delta F)$, and hence $d(RF)=-d(\Delta F)$. Both sides extend smoothly to $S$, and by the density of $S_+$, this identity holds on all of $S$. Since $S$ is connected, $\Delta F+RF$ is constant. 

It remains to prove \ref{item:potential} and
\ref{item:kahler-ricci-square}. In complex dimension $n$,
the trace identity reads
\[
\rho\wedge\omega_K^{[n-1]}
=\frac R2\,\omega_K^{[n]}.
\]
Contracting with $I^T \grad F$ and using \ref{item:kahler-rho-hamiltonian} together with $\iota_{(I^T\grad F)}\omega_K=-dF$, we obtain
\[
\frac12\,d(\Delta F)\wedge\omega_K^{[n-1]}
-dF\wedge\rho\wedge\omega_K^{[n-2]}
=-\frac R2\,dF\wedge\omega_K^{[n-1]}.
\]
Consequently, by \ref{item:kahler-box},
\[
\begin{aligned}
dF\wedge\rho\wedge\omega_K^{[n-2]}
&=\frac12\bigl(d(\Delta F)+R\,dF\bigr)
  \wedge\omega_K^{[n-1]}=-\frac F2\,dR\wedge\omega_K^{[n-1]}.
\end{aligned}
\]
On $S_+$, dividing by $F$ and using the $I^T$-invariance
of $\rho$ and $\omega_K$ yields
\begin{equation}\label{eqn:potential}
d^c\log F\wedge\rho\wedge\omega_K^{[n-2]}
=-\frac12\,d^cR\wedge\omega_K^{[n-1]}.
\end{equation}
This proves \ref{item:potential}.
Taking the exterior derivative and using
$d\rho=d\omega_K=0$ and $dd^c\log F=-\rho$, we obtain
\[
\rho\wedge\rho\wedge\omega_K^{[n-2]}
=\frac12\,dd^cR\wedge\omega_K^{[n-1]}
\]
on $S_+$. Both sides are smooth on $S$, so the identity
extends to all of $S$ by density.

We prove the last claim.
The vector field $I^T\grad F$ is not identically zero. Indeed,
otherwise $F$ would be constant and, since it is not identically
zero, it would be a positive constant. The identity
$\rho=-dd^c\log F$ would then imply $\rho=0$, contrary to $R>0$.
Furthermore, $I^T\grad F$ is Killing and holomorphic, and its
zero set is precisely $\{p\in S\mid dF(p)=0\}$.
Since $F\geq0$, one has $dF=0$ at every point of $\{F=0\}$,
and hence
\[
\{F=0\}\subseteq\{p\in S\mid dF(p)=0\}
=\{p\in S\mid I^T\grad F(p)=0\}.
\]
By the orbifold version of \cite{Kob}, which follows by applying
the manifold statement in uniformizing charts, every connected
component of the zero set of $I^T\grad F$ is a totally geodesic
complex suborbifold of positive codimension, and such connected
components are finitely many, since $S$ is compact.
The function $F$ is constant on each of these components.
Therefore, if such a component contains a point at which $F=0$,
then $F$ vanishes identically on that component. Consequently,
$\{F=0\}$ is a finite union of closed suborbifolds of real
codimension at least two. Its complement $S_+$ is therefore
connected.
\end{proof}
\begin{rmk}
Note that when $S$ is compact and $R>0$, $S \setminus S_+=F^{-1}(0) \neq \emptyset$. Assume the contrary. Then 
\[
0=\int_S d(-d^c \log F \wedge \omega_K^{[n-1]})=\int_S \rho \wedge \omega_K^{[n-1]}=\int_S \frac{R}{2} \omega_K^{[n]}>0,
\]
giving the desired contradiction.
\end{rmk}
\begin{rmk} \label{rmk:kappa}
By Proposition \ref{prop:kahler-identities},
\begin{equation*}
\kappa:= RF+\Delta F
\end{equation*}
is constant on $S$. We now derive an alternative expression for $\kappa$ on $S_+$. On $S_+$, the identity $\rho=-dd^c \log F$ gives $R=-2 \Delta \log F=-2 \frac{\Delta F}{F}+2 \frac{|dF|^2}{F^2}$. Therefore, $ \frac{|dF|^2}{F}=\Delta F +\tfrac12 RF= \Delta F +(1-\tfrac12) RF$. It follows that
\begin{equation}\label{eq:k_converse}
\kappa= \frac{|dF|^2}{F}+\frac {RF}{2}.
\end{equation}
on $S_+$. If $R>0$, this expression shows that $\kappa>0$ on $S_+$ and hence everywhere on $S$, since $\kappa$ is constant. Furthermore,
\begin{equation}\label{eqn:local_converse_first_integrals}
|dF|^2=\kappa F-\frac12RF^2,
\qquad
\Delta F=\kappa-RF.
\end{equation}
\end{rmk}
The next Proposition shows that the condition $\rho=-dd^c \log F$ is implied by a Monge-Ampere type equation. 
\begin{prop} \label{prop:MA}
Let $(S, I, g_K)$ be a compact K\"ahler orbifold surface. Assume that there exists a non-negative function $F$ and denote $S_+:=\{p \in S \ | \ F > 0 \}$. Assume that there exists a never-vanishing holomorphic volume form $\Omega$ on $S_+$. Then if $\omega_K^2= 2 F^2 \Omega \wedge \bar \Omega$ on $S_+$ then $\rho=-dd^c \log F$ on $S_+$.
\end{prop}
\begin{proof}
Let $(z_1,z_2)$ be holomorphic coordinates in $S_+$. Since $\Omega$ is holomorphic and never vanishing, $\Omega=h \, dz^1 \wedge dz^2$, for some never vanishing holomorphic function $h$ on $S_+$. The equation $\omega_K^2= 2 F^2 \Omega \wedge \bar \Omega$ can be read as 
\[
\det (g_K)= F^2 |h|^2.
\]
Then 
\[
\rho=-\tfrac12 dd^c \log (\det (g_K))= -dd^c \log F -\tfrac 12 dd^c \log |h|^2=-dd^c \log F,
\]
since $\log |h|^2$ is pluriharmonic. 
\end{proof}
\subsection{The local converse theorem}
By a GKBHE orbifold we mean an orbifold endowed with compatible orbifold tensors $(I,J,g)$ whose lifts to every uniformizing chart define a GKBHE structure. 

\begin{thm}\label{thm:local_converse}
Let $(S_+,I^T,g_K,\omega_K)$ be a connected K\"ahler orbifold
surface with Ricci form $\rho_K$ and positive scalar curvature $R_K$.
Suppose that $F$ is a strictly positive smooth orbifold function satisfying
\begin{equation}\label{eqn:local_converse_assumptions}
I^T\grad_{g_K}F \text{ is real holomorphic},
\qquad
\rho_K=-dd^c\log F.
\end{equation}
Then $\kappa:=R_KF+\Delta_{g_K}F$ is a positive constant. Let $p:M_+\to S_+$ be a principal $\mathbb T^2$-orbibundle
equipped with connection one-forms $\eta,\eta_I$ satisfying $
d\eta=0,$ and $d\eta_I=-\frac12p^*\rho_K.$ Set
\begin{equation}\label{eqn:local_converse_metric}
f:=\log\left(\frac{R_K}{8}\right),
\qquad
g:=\eta^2+\eta_I^2+e^fg_K,
\qquad
\omega_I:=\eta\wedge\eta_I+e^f\omega_K.
\end{equation}
Let $\mathcal H:=\ker\eta\cap\ker\eta_I$. Define $I$ by the horizontal
lift of $I^T$ on $\mathcal H$ and the standard complex structure on
the vertical frame dual to $(\eta,\eta_I)$. Denote this frame by
$(V_I,W_I)$, so that $IV_I=W_I$. Then $g$ is $I$-compatible,
with fundamental form $\omega_I$.
\noindent
For each $c\in\{\pm1\}$, set
\begin{equation}\label{eqn:local_converse_b_X}
b_c:=c\left(\frac{R_KF}{\kappa}-1\right),
\qquad
X_c:=-\frac{4c}{\kappa}I^T\grad_{g_K}F.
\end{equation}
Since $1+cb_c=\frac{R_KF}{\kappa}>0$, the following formulas define a smooth endomorphism $J_c$ of $TM_+$:
\begin{equation}\label{eqn:local_converse_J}
\begin{aligned}
J_cV_I&=cb_c\,IV_I+cX_c,\\
J_c(IV_I)&=-cb_c\,V_I+IX_c,\\
J_cY&=-c\,g(X_c,Y)V_I-g(IX_c,Y)IV_I-cIY+\frac{c}{1+cb_c}
\left(g(X_c,Y)IX_c-g(IX_c,Y)X_c\right),
\qquad Y\in\mathcal H.
\end{aligned}
\end{equation}
Then $(M_+,I,J_c,g)$ is a GKBHE orbifold for each $c\in\{\pm1\}$.
\end{thm}

\begin{rmk}\label{rmk:kappa-normalization}
The fact that $\kappa$ is constant is a consequence of Proposition \ref{prop:kahler-identities}. Furthermore, it is strictly positive by Remark \ref{rmk:kappa}. Its value is determined by the data $(g_K,F)$, but its numerical value is not prescribed. Indeed, for every $a>0$, the function $F_a:=aF$ satisfies the same
hypotheses, and its associated constant is $\kappa_a=a \kappa$.
The coefficients $b_c$ and $X_c$, and therefore
the reconstructed structures $(I,J_c,g)$, are unchanged.
Thus $\kappa$ encodes the normalization of $F$,
rather than an additional geometric parameter.
\end{rmk}
 We prove the theorem for $c=1$. For $c=-1$ the result will follow by Proposition \ref{prop:GK-parallel-reflection}. First, we collect some general  properties that will be needed for the proof. All the arguments below are local on the orbifold base. We therefore work throughout in uniformizing charts and on the corresponding lifted $\mathbb{T}^2$-bundles, using the same notation for the orbifold data and their lifts. Since all the constructions are invariant under the local orbifold groups and compatible with changes of charts, the resulting tensors descend to the $\mathbb T^2$-orbibundle $M_+$ over $S_+$.

Set $g^T:=e^fg_K$ and $\omega_I^T:=e^f\omega_K$. 
\medskip
\begin{lem} \label{lem:gradient}
The Hermitian structure $(I,g,f)$ is a BHE  gradient pluriclosed soliton and its Lee form is given by $\theta_I=df+2\eta$.
\end{lem}
\begin{proof}
By item \ref{item:kahler-ricci-square} in Proposition \ref{prop:kahler-identities}, the K\"ahler orbifold surface $(S_+,I^T,g_K,\omega_K)$ satisfies the hypotheses of Theorem \ref{thm:BHE-six-to-four} with $\alpha^T=0$. Consequently, the Hermitian structure $(I,g)$ on the principal $\mathbb{T}^2$-orbibundle $M_+$ is BHE, with Lee form $\theta_I=df+2\eta$.  The Bismut torsion is given by $H=\eta_I\wedge d\eta_I+H^T$, where $H^T=-d^c_{I^{T}}\omega_I^T$, and hence $\iota_{V_I}H=0$. Since $\nabla V_I=0$ by construction, it follows that $V_I$ is also $\nabla^{B,I}$-parallel. Equivalently, $\nabla^{B,I}\eta=0$. The BHE identity $\mathrm{Rc}^{B,I}+\nabla^{B,I}\theta_I=0$ (\cite[Proposition 8.14]{GFStr}) therefore reduces to $\mathrm{Rc}^{B,I}+\nabla^{B,I}df=0$, showing that $(I,g,f)$ is a  gradient pluriclosed soliton.
\end{proof}
\begin{claim} \label{claim:positivekappa}
It holds that $|X_c|^2=1-b_c^2$.
\end{claim}
\begin{proof}
Combining the definitions of $b_c$ and $X_c$ with $g^T=e^fg_K$,  Equation \eqref{eqn:local_converse_first_integrals} and $e^f=\frac{R_K}{8}$, we obtain $|X_c|^2=1-b_c^2$. \end{proof}

\begin{claim}\label{claim:db2}
It holds that $db_c=\tfrac12\iota_{X_{c}}\rho_K=-\iota_{X_{c}} d\eta_I$.
\end{claim}

\begin{proof}
Combining \eqref{eqn:contraction}, \eqref{eqn:local_converse_b_X}, and \eqref{eqn:local_converse_first_integrals}, we obtain
\[
-\frac{\kappa}{4c}\iota_{X_{c}}\rho_K
=d\left(\frac{|dF|_{g_K}^2}{F}\right)
=-\frac12d(R_KF)
=-\frac{\kappa c}{2}db_c,
\]
where in the last equality we used $b_c=c\big(\frac{R_KF}{\kappa}-1\big)$.

\end{proof}
\begin{lem} \label{lem:Udense}
The open set $\mathcal U:=\{p\in M_+\mid X_c (p)\neq0\}$
is dense on $M_+$.
\end{lem}
\begin{proof}
Suppose, by contradiction, that $\mathcal U$ is not dense. Then there exists a nonempty open subset $\mathcal V\subset M_+\setminus\overline{\mathcal U}$ on which $X_c$ vanishes identically. Since the bundle projection $p$ is open, the holomorphic vector field $X_c$ vanishes on the nonempty open subset $p(\mathcal V)\subset S_+$. By holomorphic unique continuation, $X_c$ vanishes identically on $S_+$. This implies $dF=0$. Hence $F$ is positive constant and
\[
\rho_K=-dd^c\log F=0,
\]
contradicting $R_K>0$. Therefore $\mathcal U$ is dense in $M_+$.
\end{proof}
\begin{lem}\label{lem:J_almost}
The endomorphism $J_1$ defined in \eqref{eqn:local_converse_J} is smooth, $g$-orthogonal and satisfies $J_1^2=-\operatorname{Id}$. 
\end{lem}
\begin{proof}
The endomorphism $J_1$ is globally well defined: indeed, the only potentially singular coefficient in its definition is $\frac{1}{(1+b_1)}$. However, $\tfrac{1}{1+b_1}=\frac {\kappa} {R_KF}>0$ (Remark \ref{rmk:kappa}), and hence it is smooth on $S_+$. Thus $J_1$ is smooth on $M_+$.

On the open set $\mathcal{U}=\{X_1\neq0\}\subset M_+$, set $s=|X_1|$ and choose the $I$-unitary orthonormal frame $\{e_0,\ldots,e_3\}$ defined in Equation \eqref{eqn:basis}, completed by $e_4$ and $e_5=Ie_4$. With respect to this frame, $J_1$ has the normal form \eqref{eqn:matrix_J}; hence it is $g$-orthogonal and satisfies $J_1^2=-\operatorname{Id}$. Since $\mathcal{U}$ is dense (Lemma \ref{lem:Udense}) and $J_1$ is smooth on $M_+$, these properties extend globally on $M_+$.
\end{proof}
We have therefore proved that $J_1$ is a smooth $g$-compatible almost complex structure on $M_+$. To complete the proof, it remains to verify that $N_{J_1}=0$, that $H_{J_{1}}=-H$, and that $\rho^{B,J_{1}}=0$.
These proofs rely on the following  Lemmas.
\begin{lem} \label{lem:splittingL}
Set $c=1$. The tangent bundle $TM_+$ admits an orthogonal $I$ and $J_{1}$-invariant splitting $TM_+=L_+\oplus L_-$ where $L_+=\operatorname{Ker}(I-J_1)$ and $L_-=\operatorname{Ker}(I+J_1)$.
Moreover, $L_+$ and $L_-$ have real ranks $2$ and $4$, respectively, both distributions are involutive with K\"ahler leaves, and
\[
[(1+b_1)V_I-IX_1,\Gamma(L_-)]\subseteq\Gamma(L_-),
\qquad
[(1+b_1)IV_I+X_1,\Gamma(L_-)]\subseteq\Gamma(L_-).
\]  
Furthermore $(1+b_1)V_I-IX_1, \, (1+b_1)IV_I+X_1$ are $I$-holomorphic.
\end{lem}
\begin{proof}
Define
\begin{equation} \label{eqn:L_+}
    L_+:=\operatorname{span}\{(1+b_1)V_I-IX_1,(1+b_1)IV_I+X_1\}.
\end{equation}
Since $1+b_1=\frac{R_KF}{\kappa}>0$, the subbundle $L_+$ has rank $2$ everywhere. A direct computation using $|X_1|^2=1-b_1^2$ and Equation \eqref{eqn:local_converse_J} gives $I=J_1$ on $L_+$, and therefore $I L_+=J_{1}L_+=L_+$. A further direct computation using \eqref{eqn:local_converse_J} shows that $\operatorname{Im}(I+J_1)=L_+$. Since $I+J_1$ is a skew-symmetric endomorphism,
\[
\ker(I+J_1)
=\operatorname{Im}(I+J_1)^\perp
=L_+^\perp.
\]
Define $L_-:=\ker(I+J_1)$. Then
\[
TM_+=L_+\oplus L_-,
\]
with $I=-J_1$ on $L_-$. Since $\ker(I-J_1)$ is orthogonal to $\ker(I+J_1)=L_-$, we have
\[
\ker(I-J_1)\subseteq L_-^\perp=L_+.
\]
Together with $L_+\subseteq\ker(I-J_1)$, this gives
\[
L_+=\ker(I-J_1).
\] 
In particular, both subbundles are invariant under $I$ and $J_1$. 

We first verify that $L_+$ is involutive by showing that $(1+b_1)V_I-IX_1$ and hence $(1+b_1)IV_I+X_1$ are $I$-holomorphic. 
By construction, $b_1$, $I$, $I^T$ are invariant under both $V_I$ and $IV_I$. Hence $X_1$ and $IX_1$ commute with $V_I$ and $IV_I$, and therefore
\[
(\mathcal L_{(1+b_1)V_I-IX_1})(V_I)=(\mathcal L_{(1+b_1)V_I-IX_1})(IV_I)=0.
\]
For any horizontal vector field $Y$, using $\mathcal L_{V_I}I=0$, we compute
\[
(\mathcal L_{(1+b_1)V_I-IX_1} I)(Y)
=-db_1(IY)V_I+db_1(Y)IV_I-[IX_1,IY]+I[IX_1,Y].
\]
Since $d\eta=0$, taking the scalar product with $V_I$, gives
\[
\begin{aligned}
g\bigl((\mathcal L_{(1+b_1)V_I-IX_1} I)(Y),V_I\bigr)
&=-db_1(IY)-g([IX_1,Y],IV_I)\\
&=-db_1(IY)+d\eta_I(IX_1,Y)\\
&=-db_1(IY)-d\eta_I(X_1,IY)=0
\end{aligned}
\]
where the last equality follows from Claim \ref{claim:db2}. Similarly, $g\bigl((\mathcal L_{(1+b_1)V_I-IX_1} I)(Y),IV_I\bigr)=0$.
Finally, the horizontal component of $(\mathcal L_{(1+b_1)V_I-IX_1} I)(Y)$ vanishes because $IX_1$ is transversely $I^T$-holomorphic. Indeed, $X_1$ is transversely $I^T$-holomorphic by hypothesis, and the integrability of $I^T$ implies the same property for $IX_1$. We conclude that $\mathcal L_{(1+b_1)V_I-IX_1} I=0.$ Using that $L_+$ is $I$-invariant, it has real rank two and its generators are $I$-holomorphic, it follows that $L_+$ is involutive. Since its leaves are complex curves they are in particular K\"ahler.\par
We now turn to the distribution $L_-$. Define the complex-valued one-form
\[
\beta:=\eta-\tfrac12d\log F+\sqrt{-1}\left(\eta_I-\tfrac12d^c\log F\right).
\]
Both its real and imaginary parts are closed, thus $d\beta=0$. Moreover,
\[
((1+b_1)V_I-IX_1)^\flat=(1+b_1)\left(\eta-\tfrac12d\log F\right),
\qquad
((1+b_1) IV_I+X_1)^\flat=(1+b_1)\left(\eta_I-\tfrac12d^c\log F\right).
\]
Since $1+b_1>0$ (indeed, $1+b_1=\frac{R_KF}{\kappa}$ and $\kappa$ is positive by Remark \ref{rmk:kappa}), it follows that
\[
\begin{aligned}
\ker_{\mathbb R}\beta
&=\ker\left(\eta-\tfrac12d\log F\right)
\cap\ker\left(\eta_I-\tfrac12d^c\log F\right)=L_+^\perp=L_-.
\end{aligned}
\]
Since $L_-$ is the common kernel of two closed real one-forms, it is involutive. It remains to show that its leaves are K\"ahler. Since $\beta$ vanishes on $L_-$, we have $\left.\eta\right|_{L_-}
=\left.\tfrac12d\log F\right|_{L_-}$. Therefore,
\[
\begin{aligned}
\left.d\omega_I\right|_{L_-}
&=\left.\left(e^fdf\wedge\omega_K-\eta\wedge d\eta_I\right)\right|_{L_-}\\
&=\left.\left(\tfrac18dR_K\wedge\omega_K
+\tfrac14d\log F\wedge\rho_K\right)\right|_{L_-}
=0,
\end{aligned}
\]
where the last equality follows from item \ref{item:potential} in Proposition \ref{prop:kahler-identities}. Hence the restriction of $\omega_I$ to each leaf of $L_-$ is closed, and the leaves are K\"ahler.\par
We now verify the final assertion. It suffices to prove that $\beta([(1+b_1)V_I-IX_1,L_-])=\beta([(1+b_1)IV_I+X_1,L_-])=0$. This follows immediately from the facts that $d\beta=0$, $\beta(L_-)=0$, and both $\beta((1+b_1)V_I-IX_1)$ and $\beta((1+b_1)IV_I+X_1)$ are constant.
\end{proof}
\begin{lem} \label{lem:J1_integrable}
The complex structure $J_1$ is integrable.
\end{lem}
\begin{proof}
For $c=1$, $TM_+=L_+\oplus L_-$ and $J_1=I$ on $L_+$, whereas $J_1=-I$ on $L_-$ (cf. Lemma \ref{lem:splittingL}). Hence we have
\[
T_{J_1}^{1,0}M_+=T_I^{1,0}L_+\oplus T_I^{0,1}L_-.
\]
Since the distributions $L_\pm$ are involutive and $I$ is integrable (see Lemma \ref{lem:gradient}), 
\[
[\Gamma(T_I^{1,0}L_+),\Gamma(T_I^{1,0}L_+)]
\subseteq\Gamma(T_I^{1,0}L_+)
\]
and
\[
[\Gamma(T_I^{0,1}L_-),\Gamma(T_I^{0,1}L_-)]
\subseteq\Gamma(T_I^{0,1}L_-).
\]
It therefore remains to verify the mixed brackets. 
By Lemma \ref{lem:splittingL}, the generators $(1+b_1)V_I-IX_1, \,(1+b_1)IV_I+X_1$  both preserve $L_-$ and are $I$-holomorphic. Thus, for $W\in\{(1+b_1)V_I-IX_1,(1+b_1)IV_I+X_1\}$ and $Y\in\Gamma(L_-)$,
\[
[W,Y+\sqrt{-1}IY]
=[W,Y]+\sqrt{-1}[W,IY]
=[W,Y]+\sqrt{-1}I[W,Y]
\in\Gamma(T_I^{0,1}L_-).
\]
Since $T_I^{1,0}L_+$ is generated by $\big((1+b_1)V_I-IX_1\big)-\sqrt{-1}\big( (1+b_1)IV_I+X_1\big),$
it follows that
\[
[\Gamma(T_I^{1,0}L_+),\Gamma(T_I^{0,1}L_-)]
\subseteq\Gamma(T_{J_1}^{1,0}M_+).
\]
Hence $T_{J_1}^{1,0}M_+$ is involutive, and therefore $J_1$ is integrable.
\end{proof}
\begin{proof}[Proof of Theorem \ref{thm:local_converse}]
We first prove the result for $c=1$.
By Lemma \ref{lem:J_almost}, to prove the theorem we need to show that $N_{J_{1}}=0$, $H_{J_{1}}=-H$ and $\rho^{B,J_{1}}=0$. By Lemma  \ref{lem:J1_integrable},  $J_1$ is integrable. 

Let $\omega_+$ and $\omega_-$ be the restrictions of $\omega_I$ to $L_+$ and $L_-$, extended by zero on the other summand. Then
\[
\omega_I=\omega_++\omega_-,
\qquad
\omega_{J_{1}}=\omega_+-\omega_-.
\]
Writing $d=d_++d_-$ along the two foliations, the fact that the leaves are K\"ahler (Lemma \ref{lem:splittingL}) gives
\[
d\omega_I=d_-\omega_++d_+\omega_-,
\qquad
d\omega_{J_{1}}=d_-\omega_+-d_+\omega_-.
\]
The term $d_-\omega_+$ has two arguments in $L_+$ and one in $L_-$, whereas $d_+\omega_-$ has one argument in $L_+$ and two in $L_-$. Replacing $I$ by $J_1$ on the three arguments therefore gives
\[
H_{J_{1}}=d\omega_{J_{1}}(J_1\cdot,J_1\cdot,J_1\cdot)
=-d_-\omega_+(I \cdot, I \cdot, I \cdot) -d_+\omega_-(I \cdot,I\cdot, I\cdot)=-H.
\]
Since $dH=0$ (see Lemma \ref{lem:gradient}), the structure $(I,J_1,g)$ is generalized K\"ahler. Finally, since $L_+$ has rank $2$, 
\[
\omega_I^2=\omega_-^2+2\omega_+\wedge\omega_-,
\qquad
\omega_{J_{1}}^2=\omega_-^2-2\omega_+\wedge\omega_-.
\]
Writing $\theta_I=\theta_{I,+}+\theta_{I,-}$ according to the splitting, the equation $d\omega_I^2=\theta_I\wedge\omega_I^2$ separates into
\[
d_+(\omega_-^2)=\theta_{I,+}\wedge\omega_-^2,
\qquad
d_-\omega_+\wedge\omega_-
=\theta_{I,-}\wedge\omega_+\wedge\omega_-.
\]
The same identities give the Lee equation for $J_1$, with the mixed terms changing sign on both sides. Thus
\begin{equation}\label{eqn:local_converse_Lee_plus}
\theta_{J_{1}}=\theta_I=df+2\eta.
\end{equation}
The fact that $\rho^{B,J_{1}}=0$ follows by Claim \ref{claim:GRS} together with Lemma \ref{lem:gradient}.
Define $Q=\mathrm{Id}-2\eta \otimes V_I$ and $\widehat J_1=-QJ_1Q$. It is straightforward to check that $\widehat J_1=J_{-1}$. Therefore, we may apply Proposition \ref{prop:GK-parallel-reflection} to conclude that $(I,J_{-1},g)$ is GKBHE. 
\end{proof}

\subsection{The global converse theorem}

We now extend the local converse across the zero set of $F$. The main
point is that, although the expression of $J_c$ obtained in Theorem
\ref{thm:local_converse} contains a factor $F^{-1}$, the corresponding
tensor admits a smooth extension to the whole orbifold.

\begin{thm}\label{thm:global_converse}
Let $(S,I^T,g_K,\omega_K)$ be a compact connected K\"ahler orbifold surface with Ricci form $\rho_K$ and positive scalar curvature $R_K$. Suppose that $F$ is a smooth nonnegative orbifold function, not identically zero, such that $I^T\grad_{g_K}F$ is real holomorphic on $S$ and
\begin{equation}\label{eqn:global_converse_Ricci_potential}
\rho_K=-dd^c\log F
\end{equation}
on $S_+:=\{F>0\}$. Then $S_+$ is connected and dense in $S$.
Let $p:M\to S$ be a principal $\mathbb T^2$-orbibundle equipped with connection one-forms $\eta,\eta_I$ satisfying $d\eta=0,$ and $d\eta_I=-\frac12p^*\rho_K.$ Assume that the total space $M$ is smooth and set
\begin{equation}\label{eqn:global_converse_transverse_data}
f:=\log\left(\frac{R_K}{8}\right),
\qquad
g^T:=e^fg_K,
\qquad
\omega_I^T:=e^f\omega_K,
\end{equation}
and 
\begin{equation}\label{eqn:global_converse_metric}
g:=\eta^2+\eta_I^2+g^T,
\qquad
\omega_I:=\eta\wedge\eta_I+\omega_I^T.
\end{equation}
Define $I$ by the horizontal lift of $I^T$ on
$\mathcal H:=\ker\eta\cap\ker\eta_I$ and the standard complex structure
on the vertical frame dual to $(\eta,\eta_I)$. Then $g$ is
$I$-compatible, with fundamental form $\omega_I$.

For each $c\in\{\pm1\}$, the complex structure $J_c$ provided by
Theorem~\ref{thm:local_converse} on $M_+:=p^{-1}(S_+)$ extends uniquely
to a smooth $g$-compatible complex structure $\widetilde J_c$ on $M$.
Moreover, $(M,I,\widetilde J_c,g)$ is a compact GKBHE manifold for
each $c\in\{\pm1\}$ with horizontal rank $1$.
\end{thm}

As in the previous subsection, all the computations are local and hence performed in uniformizing charts. The proof is based on the following obvious claim
\begin{claim}\label{lem:global_converse_positive_set}
The open set $S_+$ is connected and dense in $S$ and the open set $M_+$ is dense in $M$.
\end{claim}
\begin{proof}
The first assertion follows by Proposition \ref{prop:kahler-identities}. We show the second.
Since $S_+$ is dense and  $p:M\to S$ is surjective and open, also $M_+$ must be dense. 
\end{proof}

\begin{proof}[Proof of Theorem \ref{thm:global_converse}]
By item \ref{item:kahler-ricci-square} in Proposition \ref{prop:kahler-identities} and Theorem \ref{thm:BHE-six-to-four}, the Hermitian structure $(I,g)$ is BHE.

Since $S_+$ is connected, we may apply Theorem \ref{thm:local_converse} to conclude that the restriction $p:M_+\to S_+$ is a principal $\mathbb T^2$-orbibundle endowed with two complex structures, $J_1$ and $J_{-1}$ such that $(I,J_c, g)$ is GKBHE. We now verify that $J_c$ extends smoothly and uniquely to a $g$-orthogonal almost complex structure $\tilde J_c $ on $M$. 

Define on $S$ the smooth real two-form
\begin{equation}\label{eqn:global_converse_Psi}
\Psi:=dd^cF+F\rho_K
\end{equation}
On $S_+$, we have
\[
dd^c\log F
=\frac1Fdd^cF-\frac1{F^2}dF\wedge d^cF.
\]
Using $\rho_K=-dd^c\log F$, we obtain that on $S_+$
\begin{equation}\label{eqn:global_converse_Psi_on_Splus}
\Psi=\frac1F\,dF\wedge d^cF.
\end{equation}
On $M_+$, the expression of $J_c$ given in Theorem
\ref{thm:local_converse} is
\begin{align}
J_cV_I&=c\bigl(b_cIV_I+X_c\bigr),\nonumber\\
J_c(IV_I)&=-cb_cV_I+IX_c,
\label{eqn:global_converse_J_vertical}
\end{align}
and, for every horizontal vector field $Y$,
\begin{align}
J_cY=-cIY-c\,g^T(X_c,Y)V_I-g^T(IX_c,Y)IV_I+\frac{c}{1+cb_c}
\left(g^T(X_c,Y)IX_c-g^T(IX_c,Y)X_c\right).
\label{eqn:global_converse_J_singular}
\end{align}
The expressions in \eqref{eqn:global_converse_J_vertical} are smooth
on all of $M$. Thus only the last term in
\eqref{eqn:global_converse_J_singular} is potentially singular along
$\{F=0\}$, since $\frac{c}{1+cb}=\frac{\kappa c}{R_K F}$. We prove that this last term admits a smooth extension on $S$. This suffices to say that $J_c$ is smooth on the whole $M$.

With respect to $g^T=e^fg_K$, we have
\begin{equation}\label{eqn:global_converse_X_duals}
X_c^{\flat_{T}}
=-\frac{4c}{\kappa}e^fd^cF,
\qquad
(IX_c)^{\flat_{T}}
=\frac{4c}{\kappa}e^fdF.
\end{equation}
Consequently,
\begin{equation}\label{eqn:global_converse_X_wedge}
X_c^{\flat_{T}}\wedge(IX_c)^{\flat_{T}}
=
\frac{16e^{2f}}{\kappa^2}\,dF\wedge d^cF.
\end{equation}
For every horizontal vector field $Y$,
\begin{equation} \label{eqn:smooth_extension}
g^T(X_c,Y)IX_c-g^T(IX_c,Y)X_c
=
\left(
\iota_Y
\bigl(
X_c^{\flat_{T}}\wedge(IX_c)^{\flat_{T}}
\bigr)
\right)^{\sharp_{g^T}}=\frac{16e^{2f}}{\kappa^2} \big(\iota_Y \big(dF\wedge d^cF) \big)^{\sharp_{g^{T}}}
\end{equation}
where the last two equalities follow by Equation \eqref{eqn:global_converse_X_wedge}.
Using the definition of $b$ (see Equation \eqref{eqn:local_converse_b_X}), 
$R_K=8e^f$, and
\eqref{eqn:global_converse_Psi_on_Splus}, we obtain that on $S_+$
\begin{equation} \label{eqn:extension}
\frac{c}{1+cb_c}
\left(g^T(X_c,Y)IX_c-g^T(IX_c,Y)X_c\right)
=
\frac{2ce^f}{\kappa}
\left(\iota_Y\Psi\right)^{\sharp_{g^T}}.
\end{equation}
The right-hand side is smooth on all of $S$, since $\kappa$ is a
positive constant. Thus
\eqref{eqn:extension} provides a smooth extension to $S$ of the
left-hand side, initially defined on $S_+$. We may therefore extend
$J_c$ by retaining \eqref{eqn:global_converse_J_vertical} on the
vertical distribution and setting
\begin{align}
\tilde J_cY={}&-cIY-c\,g^T(X,Y)V_I-g^T(IX,Y)IV_I+
\frac{2ce^f}{\kappa}
\left(\iota_Y\Psi \right)^{\sharp_{g^T}}
\label{eqn:global_converse_J_extended}
\end{align}
for every horizontal vector field $Y$.

The endomorphism agrees with the endomorphism $J_c$ constructed on the dense open set $M_+$ (Claim \ref{lem:global_converse_positive_set}). Therefore,
\[
\tilde J_c^2=-\mathrm{Id},
\qquad
g(\tilde J_c\,\cdot,\tilde J_c\,\cdot)=g
\]
hold on all of $M$ by continuity. Thus $J_c$ extends to a smooth
$g$-compatible almost-complex structure on $M$. Its uniqueness
follows from the density of $M_+$.

The Nijenhuis tensor of $\tilde J_c$ is a smooth tensor on $M$. It vanishes
on $M_+$ by Theorem \ref{thm:local_converse}, and hence vanishes
everywhere by density. Therefore $\tilde J_c$ is integrable.
The Bismut torsion and the Bismut Ricci form are smooth tensors
determined by $\tilde J_c$ and $g$. Since on $M_+$, $H_{\tilde J_c}=-H,
$ and $\rho^{B,\tilde J_c}=0$, these identities extend on $M$ by continuity. Furthermore, $ \theta_{\tilde J_c}=\theta_{J_c}$
Thus $(I,\tilde J_c,g)$ is GKBHE.

We now prove that the horizontal rank is indeed $1$. Since $M$ is compact and $\theta_I=df + 2 \eta$ (cf. Lemma \ref{lem:gradient}) and $\theta_{\tilde J_{c}}=df + 2c \eta$ (cf. Equation \eqref{eqn:local_converse_Lee_plus}and Proposition \ref{prop:GK-parallel-reflection}), the canonical vector fields of the GKBHE structures $(I,\tilde J_c,g)$ are $(V_I, c V_I)$. Therefore, $c V_I$ has no horizontal component and $ \tilde J_c \,(c V_I)$ has horizontal component $X_c$ by \eqref{eqn:local_converse_J}. Since $X_c$ is not identically zero (otherwise $F$ would be constant and $\rho_K$ would be zero), we conclude that the horizontal rank is precisely $1$.
\end{proof}

\begin{rmk}
The compactness assumption on $S$ can be replaced by the
assumptions that $S_+=\{F>0\}$ is nonempty and connected
and that the metric $g^T=\frac{R_K}{8}g_K$ on $S$
is complete, while retaining all the other hypotheses
of the construction,
thus yielding complete GKBHE structures
on the total space for both choices of sign.
\end{rmk}
\begin{rmk} \label{rmk:nu_soliton}
As pointed out by V.~Apostolov, the identity
$\rho_K=-dd^c\log F$, with $F>0$ and
$K=I\operatorname{grad}_{g_K}F$ Killing, is precisely
the steady $\nu$-soliton equation with weight $\nu(x)=x^{-2}$.
The corresponding notion in the compact shrinking (Fano)
setting is discussed in
\cite[Section~2.3]{ApostolovLahdiliNitta2025};
see also \cite{BermanWittNystrom2014,HanLi2023,Mabuchi2003}.
With the same conventions, for a Hamiltonian holomorphic
isometric action with moment map $\mu$, the steady
$\nu$-soliton equation reads
\[
\rho=\frac12 dd^c\log\bigl(\nu\circ\mu\bigr).
\]
In our setting, the holomorphic vector field
$K=I\operatorname{grad}_{g_K}F$ satisfies
\[
\iota_K\omega_K=-dF.
\]
Consequently, $K$ preserves both $\omega_K$ and $I$, and is
therefore Killing, with Hamiltonian potential $F$.
Using $F$ as the moment coordinate and setting
$\nu(x)=x^{-2}$ for $x>0$, we obtain
\[
\frac12 dd^c\log\nu(F)
=\frac12 dd^c\log(F^{-2})
=-dd^c\log F.
\]
Thus the steady $\nu$-soliton equation is exactly
$\rho=-dd^c\log F$. This interpretation applies on
$S_+=\{F>0\}$; the weight $\nu$ is singular at $x=0$.
\end{rmk}
% --- Bibliografia ---
\bibliographystyle{plain} 
\small {
\bibliography{biblio} }
\noindent

\medskip

\noindent
B. Brienza: IMPA, Estrada Dona Castorina, 110, CEP 22460-320, Rio de Janeiro, RJ (Brasil) \\
Email:\texttt{brienza.beatrice@impa.br}
\vskip0.5em

\end{document}